\documentclass[10pt,letterpaper]{amsart}

\usepackage{amsthm,amsmath,amsfonts,amssymb,amscd} 

\usepackage{graphicx} 
\usepackage{booktabs} 
\usepackage{epstopdf} 
\usepackage{bm}       
\usepackage{multicol} 
\usepackage{multirow} 
\usepackage{subfigure} 
\usepackage{tikz}     
\usepackage{enumitem} 
\usepackage[ruled,vlined]{algorithm2e} 
\usepackage{soul}     
\usepackage{siunitx}  
\usepackage[dvipsnames]{xcolor}   
\usepackage{url}
\usepackage{hyperref}

\theoremstyle{plain} 
\newtheorem{theorem}{Theorem}[section] 
\newtheorem{lemma}[theorem]{Lemma} 
\newtheorem{corollary}[theorem]{Corollary} 
\theoremstyle{remark} 
\newtheorem{remark}[theorem]{Remark} 
\newcommand{\vect}[1]{\mathbf{#1}}
\newcommand{\divergence}{\mathrm{div}\,}
\newcommand{\tr}{\mathrm{tr}\,}
\newcommand{\grad}{\mathrm{grad}\,}
\newcommand{\Id}{\mathbb{I}} 
\newcommand{\R}{\mathbb{R}} 
\newcommand{\Rn}{\mathbb{V}} 
\newcommand{\Rsym}{\mathbb{S}} 
\newcommand{\Cel}{\mathcal{C}} 
\newcommand{\Ael}{\mathcal{A}} 

\newcommand{\Th}{\mathcal{T}_h} 
\newcommand{\Faces}{\mathcal{F}_h} 

\newcommand{\Kprod}[2]{( #1, #2 )_{K}}
\newcommand{\pKprod}[2]{\langle #1, #2 \rangle_{\partial K}}

\newcommand{\Thprod}[2]{( #1, #2 )_{\mathcal{T}_h}}
\newcommand{\pThprod}[2]{\langle #1, #2 \rangle_{\partial \mathcal{T}_{h}}}

\usepackage{amssymb}
\usepackage{amsmath}

\usepackage{lineno}

\title[HDG for poroelastic wave propagation]{Hybridizable discontinuous Galerkin methods for poroelastic wave propagation with symmetric stress approximation}

\author{Jeonghun J. Lee}
\thanks{J. J. Lee was supported by the National Science Foundation under Grant DMS-2110781.}
\address{Department of Mathematics, Baylor University, Waco, TX 76798, USA}
\email{jeonghun\_lee@baylor.edu}

\author{Manuel A. S\'anchez}
\thanks{M. A. S\'anchez was partially supported by FONDECYT Regular grant N. 1221189 and by Centro Nacional de Inteligencia Artificial CENIA, FB210017, Basal ANID Chile.}
\address{Instituto de Ingenier\'ia Matem\'atica y Computacional, Facultad de Matem\'aticas y Escuela de Ingenier\'ia, Pontificia Universidad Cat\'olica de Chile, Santiago, Chile}
\email{manuel.sanchez@uc.cl}


\begin{document}

\maketitle

\begin{abstract}
    In this paper, we develop hybridized discontinuous Galerkin (HDG) methods for poroelastic wave equations. We first rewrite the governing equations to a first-order symmetric hyperbolic system in order to use dual mixed formulations for discretization. Subsequently, we combine two HDG approaches in the discretization of the system, the $\text{HDG}+$ method for the linear elasticity equations and the $\text{LDG-H}$ method for the diffusion equations, with adjustments for the poroelastic wave equations. In our proposed HDG methods, the numerical approximation of the stress tensor is strongly symmetric and the convergence of the errors are robust for nearly incompressible materials. Upon performing static condensation, the system retains numerical trace variables solely for the solid displacement and the fluid pressure. We provide comprehensive error analyses for both the semidiscrete formulation and the Crank--Nicolson time-stepping scheme. Finally, extensive numerical examples illustrate  optimal convergence results and simulate different poroelastic wave propagation scenarios relevant in the literature.
\end{abstract}

\section{Introduction}
\label{sec:introduction}

Poroelastic wave equations describe the propagation of mechanical disturbances in fluid-saturated porous media, by the interaction between the solid matrix and the interstitial fluid. The theory was originally developed by Biot in the 1950s extending classical elasticity by incorporating fluid flow through deformable porous structures (\cite{Biot-low-frequency,Biot-high-frequency}). It has become an essential framework in fields such as geophysics, petroleum engineering (e.g., \cite{wang2000theory,coussy2004poromechanics,cheng2016poroelasticity,selvadurai2016thermo}), and biomechanics (e.g., \cite{Mow-Kuei-Lai-Armstrong:1980,Cowin:1999,ChouEtAl2016,VardakisEtAl2016,LeeAtEl:2019}).

In the theory of poroelastic waves, the medium for wave propagation is a two-phase continuum consisting of an elastic solid skeleton and a viscous fluid that occupies the pore space. The governing equations consist of coupled partial differential equations: the balance of linear momentum for the mixture of solid and fluid phases and a mass conservation equation accounting for fluid pressure and flux. The interaction of a solid skeleton and an interstitial fluid can generate richer phenomena of wave propagation than pure elastic or acoustic waves. 

Numerical methods for poroelastic waves are rich in literature including finite difference methods (e.g., \cite{Garg-Neyfeh-Good:1974,Dai:1996,Masson-Pride-Nihei:2006,Chiavassa-Lombard:2011}), finite volume methods (e.g., \cite{Lemoine-Ou-LeVeque:2013,Lemoine-Ou:2014}), finite element methods (e.g., \cite{Santos:1986b,Zienkiewicz-Shiomi:1984,Zienkiewicz:1990,laPuente-Dumbser-Kaser-Igel:2008}) but this list of previous works is not exhaustive. We refer to \cite{Carcione:2010} for more complete list of earlier works on computational methods for poroelasticity. There are more recent works on poroelastic waves with advanced finite element methods such as discontinuous Galerkin methods (\cite{Shukla-Chan-deHoop-Jaiswal:2020}), space-time finite element methods (\cite{Bause-Anselmann:2025,Kraus-Lymbery-Osthues:2025}), and virtual element methods (\cite{Chen-Liu-Zhang-Nie:2025}). 

In this paper, we consider hybridized discontinuous Galerkin (HDG) methods (\cite{Cockburn:2009a}) for poroelastic wave equations. The classical discontinuous Galerkin (DG) methods provide a finite element framework for discretizing partial differential equations (PDEs) without strictly enforcing the continuity conditions of standard finite elements; this enables their application to a wide class of PDEs and, in particular, offers advantages in wave propagation problems. However,  the computational costs of classical DG methods are expensive. HDG methods circumvent this drawback by allowing the elimination of internal element degrees of freedom by means of static condensation using the element-wise Schur complement. As a result, the number of globally coupled degrees of freedom is significantly reduced. Additionally, HDG methods often provide a way to obtain numerical solutions that fulfill important physical laws or higher-order convergence. Since wave equations often need high-order methods for high resolution of numerical solutions, and more computational cost reduction by static condensation can be made for higher order polynomials, HDG methods are well-suited for wave equations. Some previous works which devised HDG methods for various wave equations for acoustics, elastodynamics, electromagnetics can be found in \cite{Cockburn-Quenneville:2014,Nguyen-Peraire:2016,Cockburn-Fu-Sanchez-Sayas:2018,Sanchez-Du-Cockburn-Nguyen-Peraire:2022,Fernandez-Nguyen-Peraire:2018,Sanchez-Nguyen-Peraire:2017,Sanchez-Cockburn-Nguyen-Peraire:2021,Cockburn-Du-Sanchez:2023,Lee-Bolanos:2025}. 

An HDG method for time-harmonic poroelastic wave equations was first studied numerically in \cite{Barucq:2021}. For dynamic poroelasticity an HDG method was also proposed very recently in \cite{Meddahi-poroelastodynamic:2025} using a first-order system with the elastic stress, the solid velocity, the fluid velocity, and the fluid pressure as unknowns. In this paper we propose a new family of HDG methods for a first-order formulation of dynamic poroelasticity problems. In contrast to \cite{Meddahi-poroelastodynamic:2025} we use the total stress, instead of the elastic stress, as one of primary variables in the formulation because this allows more natural imposition of traction boundary conditions. The main contribution of our work is devising new efficient HDG methods for dynamic poroelasticity problems based on more recent advances in HDG methods. For this we use advanced HDG methods for elasticity (\cite{Qiu-Shen-Shi:2018,Du-Sayas:2020a,Du-Sayas:2021}, called {\it HDG+ method}) which were inspired by the work in \cite{Lehrenfeld-Schoberl:2016,Oikawa:2015,Oikawa:2016}. The main feature of $HDG+$ method is that the polynomial degrees of the hybridization variables are same or lower than the ones of internal variables without losing optimal convergence of the internal variables. A detailed comparison between existing HDG methods and our HDG methods for poroelastic wave problems will be given in Remark~\ref{rmk:comparison}. 

The paper is organized as follows. In Section~\ref{sec:preliminaries} the governing equations with first order formulation, notation for function spaces and finite element discretization, are given. In Sections~\ref{sec:semidiscrete-discretization} and \ref{sec:semidiscrete-error-analysis}, semidiscrete discretization and the error estimate of semidiscrete solutions are given. Fully discrete scheme and its error analysis are given in Section~\ref{sec:fully-discrete-error-analysis}. We include numerical test results in Section~\ref{sec:num_test}, illustrating convergence rates of errors and wave propagation in more realistic settings of physical parameters. We finish this paper with concluding remarks in Section~\ref{sec:conclusion}.

\section{Preliminaries}
\label{sec:preliminaries}

\subsection{Governing equations}\label{ss:governing-equations}

Let $\Omega \subset \mathbb{R}^{d}$ be a polygonal (if $d=2$) or
polyhedral (if $d=3$) domain. We define $\mathbb{V}$, $\mathbb{S}$ by the spaces of $d$-column vectors and $d\times d$ symmetric matrices. Let 
\begin{align*}
	\vect{u}_{s}, \vect{v}_{f}, \vect{f}: (0,\infty) \times \Omega \to \Rn    
\end{align*}
be the solid displacement, the seepage velocity, the body force, and 
\begin{align*}
	p, g:(0, \infty) \times \Omega\to \R
\end{align*}
be the pore pressure and the source/sink density function of the fluid. The governing equations of low-frequency poroelastic wave equations are
\begin{subequations}
	\label{eq:modelproblem}
	\begin{alignat}{4}
		\label{eq:modelproblem-eq1}
		\rho_{11} \ddot{\vect{u}}_{s} + \rho_{12} \dot{\vect{v}}_{f} - \divergence (\Cel \varepsilon(\vect{u}_{s}) - \alpha p \Id ) & = \vect{f} &&\qquad \text{in }\Omega,\\
		\label{eq:modelproblem-eq2}
		\rho_{12} \ddot{\vect{u}}_{s} + \rho_{22} \dot{\vect{v}}_{f} + \frac{\eta}{\kappa} \vect{v}_{f} + \grad p & = 0&&\qquad \text{in }\Omega,\\
		\label{eq:modelproblem-eq3}
		s_{0} \dot{p} + \divergence \vect{v}_{f} + \alpha \divergence \dot{\vect{u}}_{s} & = g&&\qquad \text{in }\Omega,
	\end{alignat}
\end{subequations}
where $\dot{w}$ and $\ddot{w}$ are the first and second order partial derivatives of function $w$ in time. Here $\Cel$ is the stiffness tensor of poroelastic medium, $\varepsilon$ is the symmetric gradient operator, $\mathbb{I}$ is the $d\times d$ identity matrix, $\rho_{ij}$ for $1\le i, j \le 2$ are the coefficients determined by the solid/fluid densities, $\alpha$ is the Biot--Willis coefficient, $\eta$ is the dynamic viscosity, $\kappa$ is the permeability, and $s_0 \ge 0$ is the constrained specific storage coefficient. In these equations, $\divergence$ in \eqref{eq:modelproblem-eq1} is the row-wise divergence operator whereas $\grad$ in \eqref{eq:modelproblem-eq2} is an operator mapping to $\mathbb{V}$-valued functions.

We assume that there exists $\rho_0>0$ such that 
\begin{align}
	\label{eq:rho-coercivity}
	\rho_{11}\rho_{22}-\rho_{12}^2 \ge \rho_0 >0.
\end{align}
We also assume that $\Cel$ is bounded and coercive on $L^2(\Omega; \mathbb{S})$. Defining $\Ael = \Cel^{-1}$, we also assume that there exists $C_{\Ael}>0$ such that 
\begin{align*}
	\int_{\Omega} \Ael r : s \,dx \le C_{\Ael} \| r\|_0 \|s\|_0 \qquad \forall r, s \in L^2(\Omega; \mathbb{S}) .
\end{align*}
On the contrary, we assume that $\Ael$ is positive in the sense that
\begin{multline}
	\label{eq:A-positive}
	\int_{\Omega} \Ael r : r \,dx \ge 0 \text{ for } r \in L^2(\Omega; \mathbb{S}) 
    \\
    \text{ and } \quad r = 0 \in L^2(\Omega; \mathbb{S}) \quad \text{ if } \int_{\Omega} \Ael r : r \,dx = 0 
\end{multline}
but we do not assume that $\Ael$ is coercive with a coercivity constant because this coercivity assumption does not hold for nearly incompressible materials.

Let us introduce two new variables \(\sigma := \Cel \varepsilon(\vect{u}_{s}) - \alpha p \Id\), and \(\vect{v}_{s}:=\dot{\vect{u}}_{s} .\) Then, it follows that 
\[
\Ael(\dot{\sigma} + \alpha \dot{p} \Id) - \varepsilon(\vect{v}_{s}) = 0, \quad \tr \Ael(\dot{\sigma} + \alpha \dot{p} \Id) = \divergence \vect{v}_{s} = \divergence \dot{\vect{u}}_{s},
\]
and hence we obtain the first-order symmetric system
\begin{subequations}\label{eq:1stordersystem}
	\begin{alignat}{4}
		\Ael(\dot{\sigma} + \alpha \dot{p} \Id) - \varepsilon(\vect{v}_{s}) &= 0&\qquad \text{in }\Omega,\\
		\rho_{11} \dot{\vect{v}}_{s} + \rho_{12} \dot{\vect{v}}_{f} - \divergence \sigma & = \vect{f}&\qquad \text{in }\Omega,\\
		\rho_{12} \dot{\vect{v}}_{s} + \rho_{22} \dot{\vect{v}}_{f} + \frac{\eta}{\kappa} \vect{v}_{f} + \grad p & = 0&\qquad \text{in }\Omega,\\
		s_{0} \dot{p} + \divergence \vect{v}_{f} + \alpha \tr {\Ael}(\dot{\sigma} + \alpha \dot{p} \Id)  & = g&\qquad \text{in }\Omega.
	\end{alignat}
\end{subequations}
For boundary conditions, suppose that there are two splittings of $\partial \Omega$ for traction-displacement boundary conditions and flux-pressure boundary conditions. More precisely, let $\Gamma_t, \Gamma_d, \Gamma_f, \Gamma_p \subset \partial \Omega$ be open subsets such that 
\begin{align*}
	\Gamma_t \cap \Gamma_d = \emptyset = \Gamma_f \cap \Gamma_p , \qquad 
	\partial \Omega = \overline{\Gamma_t} \cup \overline{\Gamma_d} = \overline{\Gamma_f} \cup \overline{\Gamma_p} , \qquad \Gamma_d, \Gamma_p \not = \emptyset.
\end{align*}
The boundary conditions at time $t$ are imposed by 
\begin{subequations}    
	\label{eq:boundary-conditions}
	\begin{align}
		\label{eq:boundary-condition1}
		\sigma(t) \vect{n} &= \sigma_{N}(t) \quad \text{ on } \Gamma_t, & \vect{v}_{s}(t) &= \vect{v}_{s,D}(t) \quad \text{ on } \Gamma_d, 
		\\
		\label{eq:boundary-condition2}
		\vect{v}_f(t) \cdot \vect{n} &= \vect{v}_{f,N}(t) \quad \text{ on } \Gamma_f & p(t) &= p_D(t) \quad \text{ on } \Gamma_p 
	\end{align}
\end{subequations}
where $\sigma_N,\vect{v}_{s,D}:(0,\infty) \to \mathbb{V}$, $\vect{v}_{f,N}, p_D: (0,\infty) \to \mathbb{R}$ are given boundary data and $\vect{n}$ is the unit outward normal vector field on $\partial \Omega$. 

For simplicity, we will consider only the homogeneous boundary conditions in our analysis but all discussions can be easily extended to nonhomogeneous boundary conditions.

\subsection{Notation, definitions}

For a nonnegative integer $m$, $H^m (\Omega)$ is the Sobolev space on $\Omega$ of $m$-differentiability based on the $L^2$ norm. For a finite dimensional space $\mathbb{X}$, $H^m(\Omega; \mathbb{X})$ is the Sobolev space of $\mathbb{X}$-valued functions. 
For a set $D \subset \Omega$ of positive $d$-dimensional Lebesgue measure, $\| \cdot \|_{m,D}$ and $|\cdot|_{m,D}$ are the $H^m$ norm and $H^m$ semi-norm on $D$. If $D = \Omega$, then we simply use $\| \cdot \|_m$ and $|\cdot|_m$. $H(\divergence, \Omega)$ is the set of functions in $L^2(\Omega; \R^d)$ whose divergence are in $L^2(\Omega)$, and $H(\divergence,\Omega; \Rsym)$ is the subspace of $L^2(\Omega; \Rsym)$ such that each row is in $H(\divergence, \Omega)$.

For a reflexive Banach space $\mathcal{X}$ and $(a, b) \subset \R$, $C^0 (a, b; \mathcal{X})$ denotes the set of functions $f : (a,b) \rightarrow \mathcal{X}$ which are continuous in $t \in (a,b)$. For an integer $m \geq 1$ we define 
\begin{align*}
	C^m ([0,T_0]; \mathcal{X}) := \{ f \, | \, \partial^{i}f/\partial t^{i} \in C^0(a,b;\mathcal{X}), \, 0 \leq i \leq m \},
\end{align*}
where $\partial^i f/\partial t^i$ is the $i$-th Fr\'echet derivative (see e.g., \cite{Yosida-book}). For a function $f : (a,b) \to  \mathcal{X}$, 
\begin{align*}
	\| f \|_{L^r(a,b; \mathcal{X})} := 
	\begin{cases}
		\left( \int_a^b \| f \|_\mathcal{X}^r ds \right)^{1/r}, \quad 1 \leq r < \infty, \\
		\text{esssup}_{t \in (a,b)} \| f \|_\mathcal{X}, \quad r = \infty.
	\end{cases}
\end{align*}
Define the Bochner Sobolev spaces $W^{k,r}(a, b; \mathcal{X})$ for nonnegative integer $k$ and $1 \leq r \leq \infty$ as the closure of $C^k (a,b; \mathcal{X})$ with the norm $\| f \|_{W^{k,r} \mathcal{X}} = \sum_{i=0}^k \| \partial^i f / \partial t^i \|_{L^r \mathcal{X}}$. 
We also adopt $\| f, g \|_\mathcal{X} := \| f \|_\mathcal{X} + \| g \|_\mathcal{X}$ for the norm of a Banach space $\mathcal{X}$.

Let $\mathcal{T}_{h}$ be a triangulation of $\Omega$ by simplices $K\subset \mathbb{R}^d$ ($d=2,3$). For each simplex $K$, we denote its diameter by $h_K$. We call the quantity $h:=\max_{K\in\Th}h_K$ the mesh size. We denote the diameter of a $d$-dimensional simplex $K \in \mathcal{T}_h$ by $h_K$, and define
$h:=\max_{K\in \mathcal{T}} h_K$. Let $\mathcal{F}_{h}$ denote the set of all facets of $\mathcal{T}_h$ and $\Gamma_{h}$ be the union of all facets $F \in \mathcal{F}_{h}$. 

The function spaces for variational equations are 
\begin{alignat*}{4}
	\Sigma &= \{ r \in H(\divergence, \Omega;\Rsym)\;:\; r\vect{n} = 0\text{ on }\Gamma_t \}, &\quad V_{s}&= L^{2}(\Omega;\Rn),  
	\\
	V_{f}  &= \{\vect{w}\in H(\divergence, \Omega)\;:\;\vect{w}\cdot\vect{n} = 0 \text{ on }\Gamma_{f}\}, &\quad Q&= L^{2}(\Omega).
\end{alignat*}
We denote by ${P}_k(K)$ the space of polynomials of degree less than or equal to $k\ge0$ defined on the domain $K$. For a simplex $K$, define $\mathcal{R}_k(\partial K) = \{q \in L^2(\partial K) \,: \, q|_F \in P_k(F), \forall F \in \mathcal{F}_h, F \subset \partial K\}$.  The finite-dimensional spaces for finite element discretization are
\begin{subequations}
	\label{eq:HDGspaces}
	\begin{align}
		\Sigma_{h} &= \{ \omega \in L^2(\Omega; \mathbb{S})\,:\, \omega|_{K} \in P_{k}(K; \mathbb{S}) \; \forall K \in \mathcal{T}_h \}, 
		\\
		V_{s,h} &= \{ \vect{w} \in L^2(\Omega; \mathbb{V})\,:\, \vect{w}|_{K} \in P_{k+1}(K; \mathbb{V})  \; \forall K \in \mathcal{T}_h\}, 
		\\
		V_{f,h} &= \{ \vect{w} \in L^2(\Omega; \mathbb{V})\,:\, \vect{w}|_{K} \in P_{k}(K; \mathbb{V})  \; \forall K \in \mathcal{T}_h\}, 
		\\
		Q_{h} &= \{ q \in L^2(\Omega)\,:\, q|_{K} \in P_{k}(K)  \; \forall K \in \mathcal{T}_h\}, 
		\\
		\widehat{V}_{s,h} &= \{ \widehat{\vect{w}} \in L^2(\Gamma_h; \mathbb{V})\,:\, \widehat{\vect{w}}|_{F} \in P_{k}(F; \mathbb{V})  \; \forall F \in \mathcal{F}_h, \widehat{\vect{w}}|_{\Gamma_d} = \vect{0} \}, 
		\\
		\widehat{Q}_{h} &= \{ \widehat{q} \in L^2(\Gamma_h)\,:\, \widehat{q}|_{F} \in P_{k}(F)  \; \forall F \in \mathcal{F}_h, \widehat{q}|_{\Gamma_p} = 0\} .
	\end{align}
\end{subequations}
We define $P_{\widehat{V}_s}$ and $P_{\widehat{Q}}$ by the $L^2$ orthogonal projections from $L^2(\Gamma_{h}; \mathbb{V})$ to $\widehat{V}_{s,h}$ and from $L^2(\Gamma_{h})$ to $\widehat{Q}_{h}$. 

For each element $K$ of the triangulation $\Th$, we define the following local inner products, and their associated norms:
\begin{alignat*}{8}
	\Kprod{\sigma}{r} &:= \int_K \sigma : r, 
	&&\quad
	\Kprod{\vect{v}}{\vect{w}} &:= \int_K \vect{v} \cdot \vect{w}, 
	&\quad
	\Kprod{p}{q} &&:= \int_K p q , 
	&&\quad
	\pKprod{\eta}{\mu}&&:= \int_{\partial K} \eta \mu,
	&&\quad
\end{alignat*}
The corresponding global inner products are defined by
\begin{align*}
	\Thprod{\sigma}{r} :=
	\sum_{K \in \Th} \Kprod{\sigma}{r}, \quad
	\Thprod{\vect{v}}{\vect{w}} := \sum_{K \in \Th} \Kprod{\vect{v}}{\vect{w}}, \quad
	\pThprod{\eta}{\mu}:=
	\sum_{K \in \Th} \pKprod{\eta}{\mu}.
\end{align*}

\section{Semidiscrete space discretization}
\label{sec:semidiscrete-discretization}

{
In this section, we present the semidiscrete discretization of the first-order system \eqref{eq:1stordersystem} using a Hybridizable Discontinuous Galerkin approach in space and continuous in time. We begin by introducing the detailed HDG formulation in Section \ref{ss:HDGformulation}, including the definition of the numerical traces and the assumptions on the stabilization functions. We also prove here the consistency property of the semidiscrete approximation. Since the resulting system is a differential-algebraic equation, we subsequently discuss the necessary compatibility conditions and approximation properties for the initial data in Section \ref{ss:initialdata} to guarantee the well-posedness of the problem. Finally, in Section \ref{ss:stability}, we establish a general energy equality for the semidiscrete scheme, which will serve as the foundational tool for the a priori error estimates presented in the subsequent sections.}
\subsection{The HDG formulation}\label{ss:HDGformulation}
The semidiscrete problem reads as find $(\sigma_{h}, \vect{v}_{s,h}, \vect{v}_{f,h}, p_{h}) \in C^1(0,T; \Sigma_{h} \times V_{s,h} \times V_{f,h} \times Q_{h} )$, and $(\widehat{\vect{v}}_{s,h}, \widehat{p}_{h}) \in C^0(0,T;  \widehat{V}_{s,h}\times \widehat{Q}_{h})$
%
%
such that 
\begin{subequations}
	\label{eq:semidiscrete-eqs}
	\begin{align}
		\label{eq:semidiscrete-eq1}
		\Thprod{\Ael(\dot{\sigma}_{h} + \alpha \dot{p}_{h} \Id)}{r} +\Thprod{\divergence r}{\vect{v}_{s,h}} -\pThprod{\widehat{\vect{v}}_{s,h}}{r \vect{n}} &= 0 ,
		\\ 
		\label{eq:semidiscrete-eq2}
		\Thprod{\rho_{11}\dot{\vect{v}}_{s,h} + \rho_{12}\dot{\vect{v}}_{f,h}}{\vect{w}_{s}} + \Thprod{\sigma_{h}}{\grad \vect{w}_{s}}&
        \\
        \notag
        \quad - \pThprod{\widehat{\sigma_{h} \vect{n}}}{\vect{w}_{s}}  &= \Thprod{\vect{f}}{\vect{w}_{s}}, 
		\\
		\label{eq:semidiscrete-eq3}
		\Thprod{\rho_{12}\dot{\vect{v}}_{s,h} + \rho_{22}\dot{\vect{v}}_{f,h}}{\vect{w}_{f}} + \Thprod{\frac{\eta}{\kappa} \vect{v}_{f,h}}{\vect{w}_{f}} - \Thprod{p_{h} }{\divergence\vect{w}_{f}} &
        \\
        \notag
        \quad +\pThprod{\widehat{p}_{h}}{\vect{w}_{f}\cdot\vect{n}} &= 0, 
		\\
		\label{eq:semidiscrete-eq4}
		\Thprod{s_{0}\dot{p}_{h}}{q} + \Thprod{\Ael(\dot{\sigma}_{h}+\alpha \dot{p}_{h}\Id)}{\alpha q\Id} - \Thprod{\vect{v}_{f,h}}{\grad q}&
        \\
        \notag
        \quad + \pThprod{\widehat{\vect{v}_{f,h}\cdot \vect{n}}}{q} &= \Thprod{g}{q},
		\\
		\label{eq:semidiscrete-eq5}
		\pThprod{\widehat{\sigma_{h} \vect{n}} }{\widehat{\vect{w}}_s} &= 0,
		\\
		\label{eq:semidiscrete-eq6}
		\pThprod{\widehat{\vect{v}_{f,h}\cdot \vect{n}}}{\widehat{q}} &= 0
	\end{align}
	for any $(r, \vect{w}_s, \vect{w}_f, q, \widehat{\vect{w}}_s, \widehat{q}) \in \Sigma_{h} \times V_{s,h} \times V_{f,h} \times Q_{h} \times \widehat{V}_{s,h}\times \widehat{Q}_{h}$
	and for all $0<t<T$. The HDG numerical traces $\widehat{\sigma_h \vect{n}}$ and $\widehat{\vect{v}_{f,h} \cdot \vect{n}}$ are defined by
	\begin{align}
		\label{eq:numerical-traction}
		\widehat{\sigma_{h} \vect{n}} & := \sigma_h \vect{n} - \tau_{s} (P_{\widehat{V}_s} \vect{v}_{s,h} - \widehat{\vect{v}}_{s,h}), 
		\\
		\label{eq:numerical-flux}
		\widehat{\vect{v}_{f,h} \cdot\vect{n}}& := \vect{v}_{f,h}\cdot\vect{n} + \tau_{f} (p_{h} - \widehat{p}_{h}), 
	\end{align}
	where $\tau_{s}, \tau_{f}$ are stabilization functions defined on the skeleton $\partial \Th$, satisfying the following assumptions. 
\end{subequations}

\vskip2mm
\noindent 
\textbf{Assumptions on stabilization functions $\tau_{s}$ and $\tau_{f}$.} We will assume these three conditions, which are necessary for well-posedness, in the rest of this paper: 
\begin{itemize}
	\item[(A1)] Suppose that $\tau_s, \tau_f \in \prod_{K \in \mathcal{T}_h} \mathcal{R}_0(\partial K)$, where $\mathcal{R}_0(\partial K)$ is the space of face-wise constant functions on $\partial K$. Assume that $\tau_f|_{\partial K} > 0$ on at least one face $F \subset \partial K$ for all $K \in \mathcal{T}_h$
	\item[(A2)] There exists $c_1, c_2 >0$ independent of mesh sizes such that $c_1 \le \max \tau_f|_{\partial K} \le c_2$ for all $K \in \mathcal{T}_h$
	\item[(A3)] There exists $c_1, c_2 >0$ independent of mesh sizes such that $c_1 h_K^{-1} \le \tau_s|_{\partial K} \le c_2 h_K^{-1}$ for all $K \in \mathcal{T}_h$
\end{itemize}

We omit the proof of the well-posedness of the semidiscrete problem \eqref{eq:semidiscrete-eqs}. Instead, we provide in Section \ref{sec:fully-discrete-error-analysis} the well-posedness of a fully discrete scheme derived from these equations. Next, we state a consistency result for the scheme \eqref{eq:semidiscrete-eqs}.

\begin{lemma}[Consistency]
	\label{lemma:consistency}
	Let \( (\sigma, \vect{v}_s, \vect{v}_f, p) \in C^1(0,T; \Sigma_N \times H^1(\Omega; \mathbb{V}) \times W_N \times H^1(\Omega)) \) be a solution of \eqref{eq:1stordersystem} with the homogeneous boundary conditions given by \eqref{eq:boundary-conditions}. Then, 
	
	\begin{subequations}
		\label{eq:variational-eqs}
		\begin{align}
			\label{eq:variational-eq1}
			\Thprod{\Ael(\dot{\sigma} + \alpha \dot{p} \Id)}{r} +\Thprod{\divergence r}{\vect{v}_{s}} -\pThprod{ \vect{v}_{s}}{r \vect{n}} &= 0 ,
			\\ 
			\label{eq:variational-eq2}
			\Thprod{\rho_{11}\dot{\vect{v}}_{s} + \rho_{12}\dot{\vect{v}}_{f}}{\vect{w}_{s}} + \Thprod{\sigma}{\grad \vect{w}_{s}} &
            \\
            \notag
            - \pThprod{\sigma \vect{n} - \tau_s ({P_{\widehat{V}_s}} \vect{v}_s - \widehat{\vect{v}}_s) }{\vect{w}_{s}} & = \Thprod{\vect{f}}{\vect{w}_{s}}, 
			\\
			\label{eq:variational-eq3}
			\Thprod{\rho_{12}\dot{\vect{v}}_{s} + \rho_{22}\dot{\vect{v}}_{f}}{\vect{w}_{f}} + \Thprod{\frac{\eta}{\kappa} \vect{v}_{f}}{\vect{w}_{f}} - \Thprod{p }{\divergence\vect{w}_{f}} +\pThprod{ p}{\vect{w}_{f}\cdot\vect{n}} & = 0, 
			\\
			\label{eq:variational-eq4}
			\Thprod{s_{0}\dot{p}}{q} + \Thprod{\Ael(\dot{\sigma}+\alpha \dot{p}\Id)}{\alpha q\Id} - \Thprod{\vect{v}_{f}}{\grad q} &
            \\
            \notag
            + \pThprod{\vect{v}_{f}\cdot \vect{n} + \tau_f (p - p)}{q} & = \Thprod{g}{q},
			\\
			\label{eq:variational-eq5}
			\pThprod{\sigma \vect{n} - \tau_s ({P_{\widehat{V}_s}} \vect{v}_s - {\vect{v}}_s) }{\widehat{\vect{w}}_s} & = 0,
			\\
			\label{eq:variational-eq6}
			\pThprod{\vect{v}_{f}\cdot \vect{n} + \tau_f (p - \widehat{p})}{\widehat{q}} & = 0
		\end{align}
	\end{subequations}
	for any $(r, \vect{w}_s, \vect{w}_f, q, \widehat{\vect{w}}_s, \widehat{q}) \in \Sigma_{h} \times V_{s,h} \times V_{f,h} \times Q_{h} \times \widehat{V}_{s,h}\times \widehat{Q}_{h}$ for all $0<t<T$ where $\widehat{\vect{v}}_s = \vect{v}_s|_{\Gamma_h}$, $\widehat{p} = p|_{\Gamma_h}$. 
\end{lemma}
\begin{proof}
	The equations \eqref{eq:variational-eq1}, \eqref{eq:variational-eq3}, \eqref{eq:variational-eq4} follow by the integration by parts and \eqref{eq:1stordersystem}. \eqref{eq:variational-eq2} and \eqref{eq:variational-eq5} hold because $\pThprod{{P_{\widehat{V}_s}} \vect{v}_s - \widehat{\vect{v}}_s}{\widehat{w}_s} =0$ and $\sigma \vect{n}$ is continuous on $\mathcal{F}_h$. Finally, \eqref{eq:variational-eq6} holds because $\vect{v}_f \cdot \vect{n}$ is continuous on $\mathcal{F}_h$. 
\end{proof}

\begin{remark}
    \label{rmk:comparison}
	Here we briefly summarize the main differences between the HDG method in \cite{Meddahi-poroelastodynamic:2025} and our method. First, the primary variables in \cite{Meddahi-poroelastodynamic:2025} are the elastic stress tensor $\Cel \varepsilon(\vect{u}_s)$, $\vect{v}_s$, $\vect{v}_f$, $p$. When $\Gamma_t$ and $\Gamma_p$ intersect, imposing the traction boundary condition \eqref{eq:boundary-condition1} to the primary variable $\Cel \varepsilon(\vect{u}_s)$ needs an extra concern because $p$ values are given on $\Gamma_p \cap \Gamma_t$. In contrast, we use the total stress tensor $\Cel \varepsilon(\vect{u}_s) - \alpha p \mathbb{I}$, so imposing the traction boundary condition is straightforward. Second, two vector variables $\widehat{\vect{v}}_s, \widehat{\vect{v}}_f \in P_{l+1}(F; \mathbb{V})$, $l\ge 0$, are used for hybridization in \cite{Meddahi-poroelastodynamic:2025} with the polynomial spaces $P_l(K; \mathbb{S}) \times P_{l+1}(K; \mathbb{V}) \times P_{l+1}(K; \mathbb{V}) \times P_{l}(K)$ for internal variables $\Cel \varepsilon(\vect{u}_s)$, $\vect{v}_s$, $\vect{v}_f$, $p$, so the errors of the internal variables converge at a rate of $O(h^{l+1})$. In our method \eqref{eq:HDGspaces}, we use one vector and one scalar variables $\widehat{\vect{v}}_s$, $\widehat{p}$ as hybridization variables. For a given polynomial degree used in hybridization (i.e., $k=l+1$), the proposed method in this paper results in a reduced number of globally coupled degrees of freedom. In addition, the internal variables are approximated with polynomial degrees that are equal to or higher than those of the hybridization variables. This choice leads to $O(h^{k+1})$ or higher order convergence rates for the internal variables.
\end{remark}

\subsection{Initial data compatibility}\label{ss:initialdata}
The system \eqref{eq:semidiscrete-eqs} is a differential algebraic equation (DAE), not a system of ordinary differential equations because \eqref{eq:semidiscrete-eq5}, \eqref{eq:semidiscrete-eq6} are algebraic equations. In the theories of DAE (cf. \cite{Kunkel-Mehrmann:2006}), a compatibility of initial data is essential for well-posedness of the problem. 

For initial data $(\sigma(0), \vect{v}_{s}(0), \vect{v}_{f}(0), p(0))$ we will assume that numerical initial data $(\sigma_{h}(0)$, $\vect{v}_{s,h}(0)$, $\vect{v}_{f,h}(0)$, $p_{h}(0)$, $\widehat{\vect{v}}_{s,h}(0)$, $\widehat{p}_{h}(0))$ satisfy the compatibility conditions
\begin{subequations}
	\begin{align}
		\label{eq:initial-data-constraint1}
		\sigma_h(0) \vect{n} - \tau_{s} ({P_{\widehat{V}_s}} \vect{v}_{s,h}(0) - \widehat{\vect{v}}_{s,h}(0)) &= 0 \quad \text{ on } \mathcal{F}_h, 
		\\
		\label{eq:initial-data-constraint2}
		\vect{v}_{f,h}(0)\cdot\vect{n} + \tau_{f} (p_{h}(0) - \widehat{p}_{h}(0)) &= 0  \quad \text{ on } \mathcal{F}_h, 
	\end{align}
	and approximation properties 
	\begin{align}
		\label{eq:initial-data-approximation1}
		\|\sigma(0) - \sigma_{h}(0) \|_{L^2} + \|\vect{v}_{s} (0) - \vect{v}_{s,h}(0) \|_{L^2} &
        \\
        \notag 
        + h^{1/2}\|\vect{v}_{s} (0) - \widehat{\vect{v}}_{s,h}(0) \|_{{L^2(\mathcal{F}_h}) }  &\le Ch^m |\sigma(0), \vect{v}_{s}(0)|_{m}, 
		\\
		\label{eq:initial-data-approximation2}
		\|\vect{v}_{f}(0) - \vect{v}_{f,h}(0) \|_{L^2} + \|p(0) - p_h(0)\|_{L^2} &
        \\
        \notag
        + h^{1/2} \|p(0) - \widehat{p}_h(0)\|_{L^2(\mathcal{F}_h)} & \le Ch^m |\vect{v}_{f}(0), p(0)|_{m} 
	\end{align}
	for $1 \le m \le k+1$.
\end{subequations}

We can obtain numerical initial data satisfying \eqref{eq:initial-data-constraint1}, \eqref{eq:initial-data-constraint2}, \eqref{eq:initial-data-approximation1}, \eqref{eq:initial-data-approximation2} by solving two suitable elliptic problems as follows: 

Let $\vect{v}_f(0) \in L^1(\Omega; \mathbb{V})$, $\divergence \vect{v}_f(0), p(0) \in L^1(\Omega)$, $p(0)|_{\mathcal{F}_h} \in L^1(\Gamma_h)$ be given. To find numerical initial data satisfying \eqref{eq:initial-data-constraint2}, find $(\vect{v}_{f,h}(0), p_h(0), \widehat{p}_h(0)) \in V_{f,h} \times Q_{h} \times \widehat{Q}_h$ such that 
\begin{align*}
	\Thprod{\vect{v}_{f,h}(0)}{\vect{w}_{f}} - \Thprod{p_{h}(0)}{\divergence \vect{w}_{f}} + \pThprod{\widehat{p}_h(0)}{\vect{w}_{f}\cdot \vect{n}} 
	& = L_{1,1}(\vect{w}_f) && \quad \forall \vect{w}_f\in V_{f,h},
	\\
	-\Thprod{\vect{v}_{f,h}(0)}{\grad q} + \pThprod{\widehat{\vect{v}}_{f,h}(0)\cdot \vect{n}}{q} 
	&= L_{1,2}(q) && \quad \forall q \in Q_{h},
	\\
	\pThprod{ \widehat{\vect{v}}_{f,h}(0)\cdot \vect{n} }{\widehat{q}} 
	&= 0 && \quad \forall q \in \widehat{Q}_{h},
\end{align*}
where the HDG numerical flux is $\widehat{\vect{v}}_{f,h}(0)\cdot \vect{n}=\vect{v}_{f,h}(0)\cdot\vect{n} + \tau_{f} (p_{h}(0) - \widehat{p}_{h}(0))$ on $\partial \Th$, and the linear functionals are given by
\[
L_{1,1}(\vect{w}_f) = \Thprod{\vect{v}_{f}(0)}{\vect{w}_{f}} - \Thprod{p(0)}{\divergence \vect{w}_{f}} + \pThprod{p(0)}{\vect{w}_{f}\cdot \vect{n}} ,\quad
L_{1,2} (q) = \Thprod{\divergence\vect{v}_{f}(0)}{q}.
\]
Similarly, for given $\sigma(0) \in L^1(\Omega; \mathbb{S})$, $\divergence \sigma(0), v_{s}(0) \in L^1(\Omega; \mathbb{V})$, $v_{s}(0)|_{\mathcal{F}_h} \in L^1(\Gamma_h; \mathbb{V})$, find $(\sigma_h(0), \vect{v}_{s,h}(0), \widehat{\vect{v}}_{s,h}(0)) \in \Sigma_h \times V_{s,h} \times \widehat{V}_{s,h}$ such that 
\begin{align*}
	\Thprod{\Ael \sigma_{h}(0) }{r} +\Thprod{\divergence r}{\vect{v}_{s,h}} -\pThprod{\widehat{\vect{v}}_{s,h}(0)}{r \vect{n}} &= L_{2,1}(r)&& \quad \forall r\in \Sigma_h, 
    \\
	-\Thprod{\sigma_{h}}{\grad \vect{w}_{s}} - \pThprod{ \widehat{\sigma}_h(0)\vect{n}}{\vect{w}_{s}} &= 
	L_{2,2}(\vect{w}_s) &&\quad \forall \vect{w}_s\in V_{s,h},
    \\
	\pThprod{\widehat{\sigma}_h(0)\vect{n}}{\widehat{\vect{w}}_s} &= 0 && \quad\forall \widehat{\vect{w}}_{s,h} \in \widehat{V}_{s,h},
\end{align*}
where the HDG numerical flux is $\widehat{\sigma}_h(0)\vect{n} = \sigma_h(0) \vect{n} - \tau_{s} ({P_{\widehat{V}_s}} \vect{v}_{s,h}(0) - \widehat{\vect{v}}_{s,h}(0))$ and the linear functionals are
\[
L_{2,1}(r) = \Thprod{\Ael \sigma(0) }{r} +\Thprod{\divergence r}{\vect{v}_{s}(0)} -\pThprod{\vect{v}_{s}(0)}{r \vect{n}} ,\quad 
L_{2,2}(\vect{w}_s) = \Thprod{\divergence \sigma (0)}{\vect{w}_{s}}.
\]

These triplets of initial data obviously satisfy \eqref{eq:initial-data-constraint1}, \eqref{eq:initial-data-constraint2}. One can check that \eqref{eq:initial-data-approximation1}, \eqref{eq:initial-data-approximation2} are satisfied by the standard error analyses of second order elliptic problems and linear elasticity problems as long as the initial data satisfy the necessary regularities in \eqref{eq:initial-data-constraint1}, \eqref{eq:initial-data-constraint2} (cf. \cite{Cockburn:2010,Qiu-Shen-Shi:2018,Du-Sayas:2020a}).

\subsection{Energy equality for the semidiscrete HDG formulation}\label{ss:stability}
%
We conclude this section by presenting an energy equality of the semidiscrete scheme and for general right-hand sides, which will be utilized in the subsequent error analysis.
\begin{lemma}[General energy identity]\label{lemma:energy}
	Let  $(\sigma_{h}, \vect{v}_{s,h}, \vect{v}_{f,h}, p_{h}) \in C^1(0,T; \Sigma_{h} \times V_{s,h} \times V_{f,h} \times Q_{h} )$ and  $(\widehat{\vect{v}}_{s,h}, \widehat{p}_{h}) \in C^0(0,T;  \widehat{V}_{s,h}\times \widehat{Q}_{h})$ be a solution of  
	\begin{subequations}
		\label{eq:general-semidiscrete-eqs}
		\begin{align}
			\label{eq:general-semidiscrete-eq1}
			\Thprod{\Ael(\dot{\sigma}_{h} + \alpha \dot{p}_{h} \Id)}{r} +\Thprod{\divergence r}{\vect{v}_{s,h}} -\pThprod{\widehat{\vect{v}}_{s,h}}{r \vect{n}} &= F_1(r) ,
			\\ 
			\label{eq:general-semidiscrete-eq2}
			\Thprod{\rho_{11}\dot{\vect{v}}_{s,h} + \rho_{12}\dot{\vect{v}}_{f,h}}{\vect{w}_{s}} + \Thprod{\sigma_{h}}{\grad \vect{w}_{s}}&
            \\
            \notag 
            - \pThprod{\widehat{\sigma_{h} \vect{n}}}{\vect{w}_{s}} & = F_2(\vect{w}_{s}), 
			\\
			\label{eq:general-semidiscrete-eq3}
			\Thprod{\rho_{12}\dot{\vect{v}}_{s,h} + \rho_{22}\dot{\vect{v}}_{f,h}}{\vect{w}_{f}} + \Thprod{\frac{\eta}{\kappa} \vect{v}_{f,h}}{\vect{w}_{f}} &
            \\
            \notag 
            - \Thprod{p_{h} }{\divergence\vect{w}_{f}} +\pThprod{\widehat{p}_{h}}{\vect{w}_{f}\cdot\vect{n}} & = F_3(\vect{w}_f), 
			\\
			\label{eq:general-semidiscrete-eq4}
			\Thprod{s_{0}\dot{p}_{h}}{q} + \Thprod{\Ael(\dot{\sigma}_{h}+\alpha \dot{p}_{h}\Id)}{\alpha q\Id} - \Thprod{\vect{v}_{f,h}}{\grad q} &
            \\
            \notag 
            + \pThprod{\widehat{\vect{v}_{f,h}\cdot \vect{n}}}{q} & = F_4(q),
			\\
			\label{eq:general-semidiscrete-eq5}
			\pThprod{\sigma_h \vect{n} - \tau_{s} ({P_{\widehat{V}_s}} \vect{v}_{s,h} - \widehat{\vect{v}}_{s,h}) }{\widehat{\vect{w}}_s} & = F_5(\widehat{\vect{w}}_s),
			\\
			\label{eq:general-semidiscrete-eq6}
			{-}    \pThprod{\vect{v}_{f,h}\cdot\vect{n} + \tau_{f} (p_{h} - \widehat{p}_{h})}{\widehat{q}} & = F_6(\widehat{q})
		\end{align}
	\end{subequations}
	for any $(r, \vect{w}_s, \vect{w}_f, q, \widehat{\vect{w}}_s, \widehat{q}) \in \Sigma_{h} \times V_{s,h} \times V_{f,h} \times Q_{h} \times \widehat{V}_{s,h}\times \widehat{Q}_{h}$ and  with general bounded linear functionals $F_i(\cdot)(t), 1 \le i \le 6$, $0\le t \le T$ on $\Sigma_h$, $V_{s,h}$, $V_{f,h}$, $Q_h$, $\widehat{V}_{s,h}$, $\widehat{Q}_h$. Then, 
	%
	\begin{align}
		\label{eq:semidiscrete-energy-identity}
		X(t)^2 - X(0)^2 + 2Y(t)^2 = 2 \int_0^t &\left( (F_1(\sigma_h(s)) + F_2(\vect{v}_{s,h}(s)) + F_3(\vect{v}_{f,h} (s)) \right.  
		\\
		\notag
		& \;+ \left. F_4(p_h(s)) + F_5(\widehat{\vect{v}}_{s,h}(s)) + F_6(\widehat{p}_h(s))\right) \, ds
	\end{align}
	for $X(t), Y(t) \ge 0$ defined by  
	\begin{align*}
		X(t)^2 := &\Thprod{\Ael({\sigma}_{h}(t) + \alpha {p}_{h}(t) \Id)}{{\sigma}_{h}(t) + \alpha {p}_{h}(t) \Id} +\Thprod{s_{0}{p}_{h}(t)}{p_h(t)}
		\\
		& + \Thprod{\rho_{11}{\vect{v}}_{s,h}(t)}{\vect{v}_{s,h}(t)} + 2\Thprod{\rho_{12}{\vect{v}}_{f,h}(t)}{\vect{v}_{s,h}(t)}  + \Thprod{\rho_{22}{\vect{v}}_{f,h}(t)}{\vect{v}_{f,h}(t)} ,
		\\
		Y(t)^2 :=&\int_0^t \Big( \Thprod{\frac{\eta}{\kappa} \vect{v}_{f,h}(s)}{\vect{v}_{f,h}(s)}  + \pThprod{\tau_{s} ({P_{\widehat{V}_s}} \vect{v}_{s,h}(s) - \widehat{\vect{v}}_{s,h}(s))}{{P_{\widehat{V}_s}}\vect{v}_{s,h}(s)-\widehat{\vect{v}}_{s,h}(s) } 
		\\
		&\qquad +  \pThprod{\tau_{f} (p_{h}(s) - \widehat{p}_{h}(s))}{p_h(s)-\widehat{p}_h(s)} \Big) \,ds. 
	\end{align*}
\end{lemma}
\begin{proof}
	We test the system \eqref{eq:general-semidiscrete-eqs} with $r = \sigma_h, \; \vect{w}_s = \vect{v}_{s,h}, \; \vect{w}_f = \vect{v}_{f,h}, \; q = p_h, \; \widehat{\vect{w}}=\widehat{\vect{v}}_{s,h}, \; \widehat{q} = \widehat{p}_h$, in \eqref{eq:general-semidiscrete-eqs}. Then, we add \eqref{eq:general-semidiscrete-eq1}, \eqref{eq:general-semidiscrete-eq2}, and \eqref{eq:general-semidiscrete-eq5}, integrate by parts to obtain
	\begin{align}
		\nonumber
		\Thprod{\Ael(\dot{\sigma}_{h} + \alpha \dot{p}_{h} \Id)}{\sigma_h}+ \Thprod{\rho_{11}\dot{\vect{v}}_{s,h} + \rho_{12}\dot{\vect{v}}_{f,h}}{\vect{v}_{s}}  + &\pThprod{\vect{v}_{s,h} - \widehat{\vect{v}}_{s,h}}{ \sigma_h \vect{n} - \widehat{\sigma_h \vect{n} }} 
		\\
	\label{eq:energy-equality-proof1}
		&= F_1(\sigma_h) + F_2(\vect{v}_{s,h}) + F_5(\widehat{\vect{v}}_{s,h}).
	\end{align}
	Using the definition of the numerical flux $\widehat{\sigma_h \vect{n} }$ and the definition of ${P_{\widehat{V}_s}}$ we have that 
	\begin{align*}
		\pThprod{\vect{v}_{s,h} - \widehat{\vect{v}}_{s,h}}{ \sigma_h \vect{n} - \widehat{\sigma_h \vect{n} }} = 
		\pThprod{ P_{\widehat{V}_s}\vect{v}_{s,h} - \widehat{\vect{v}}_{s,h} }{  \tau_{s}\left( {P_{\widehat{V}_s}}\vect{v}_{s,h} - \widehat{\vect{v}}_{s,h}  \right)} .
	\end{align*}
	Similarly, we add \eqref{eq:general-semidiscrete-eq3}, \eqref{eq:general-semidiscrete-eq4}, and \eqref{eq:general-semidiscrete-eq6}, integrate by parts to obtain
	\begin{align}
		\nonumber
		\Thprod{\rho_{12}\dot{\vect{v}}_{s,h} + \rho_{22}\dot{\vect{v}}_{f,h}}{\vect{v}_{f}} &+ \Thprod{\frac{\eta}{\kappa} \vect{v}_{f,h}}{\vect{v}_{f,h}}+
		\Thprod{s_{0}\dot{p}_{h}}{q} + \Thprod{\Ael(\dot{\sigma}_{h}+\alpha \dot{p}_{h}\Id)}{\alpha q\Id} \\
		\label{eq:energy-equality-proof2}
		+ \pThprod{p_h-\widehat{p}_h}{\tau_f(p_h - \widehat{p}_h)} & = F_{3}(\vect{v}_{f,h}) + F_{4}(p_h) + F_{6}(\widehat{p}_h) .
	\end{align}
	According to the definition of $X(t)$, we observe that adding \eqref{eq:energy-equality-proof1} and \eqref{eq:energy-equality-proof2} gives
	\begin{align*}
		\frac{1}{2}\frac{d}{dt} \left( X(t)^{2} + 2Y(t)^{2}\right)   =  F_1(\sigma_h) + F_2(\vect{v}_{s,h}) +F_{3}(\vect{v}_{f,h}) + F_{4}(p_h) +  F_5(\widehat{\vect{v}}_{s,h})+F_{6}(\widehat{p}_h)
	\end{align*}
	from which we obtain the energy identity \eqref{eq:semidiscrete-energy-identity} after integrating over $t$.
\end{proof}

\begin{corollary}[Energy identity]
	Consider the solution of \eqref{eq:semidiscrete-eqs} and the definitions of $X(t)$ and $Y(t)$ as in the previous lemma. Then
	\[
	X(t)^{2} - X(0)^{2} + 2Y(t)^{2} = 2 \int_0^{t}\left( \Thprod{\vect{f}(s)}{\vect{v}_{s,h}(s)} + \Thprod{g(s)}{p_{h}(s)}\right) ds.
	\]
	
\end{corollary}

\section{Semi-discrete Error Analysis}\label{sec:semidiscrete-error-analysis}

{
In this section, we provide full details of the \textit{a priori} error estimates of the semidiscrete HDG approximation formulated in Section \ref{ss:HDGformulation}. We begin in Section \ref{ss:interpolation} by defining the appropriate local HDG projection operators and stating their fundamental approximation properties. Subsequently, in Section \ref{subsec:error-equations}, we decompose the total numerical error into interpolation and projection components to systematically derive the associated error equations. Finally, in Section \ref{subsec:energy-estimates}, we present the main result of this section, Theorem \ref{thm:semidiscrete-error-estimate}, which states the optimal semidiscrete error estimates.
}
\subsection{Projection operators}\label{ss:interpolation}
In the following theorems, we define the local HDG projection operators $(\Pi_{\sigma,K}, \Pi_{\vect{v}_{s},K})$ and $(\Pi_{\vect{v}_{f},K}, \Pi_{p,K})$ and state their result on existence and approximation properties, which will be used later. We also provide references to these results.
\begin{theorem}
	\label{thm:elasticity-interpolation}
	For $k\ge 1$, $r>\frac 12$, $K \in \mathcal{T}_h$, and  $\tau_{s} \in \mathcal{R}_k(\partial K)$ which is uniformly bounded from above and below by positive constants, there exists a triple of linear operators $(\Pi_{\sigma,K}, \Pi_{\vect{v}_s, K},R_K)$ such that
	\begin{alignat*}{3}
		\Pi_{\sigma,K} \times \Pi_{\vect{v}_s,K} &: H^r(K; \mathbb{S}) \times H^r(K; \mathbb{V}) \to {P}_k(K; \mathbb{S}) \times {P}_{k+1}(K;\mathbb{V}), 
		\\
		R_K &: H^r(K; \mathbb{S}) \times H^r(K; \mathbb{V}) \to \mathcal{R}_k(\partial K; \mathbb{V})
	\end{alignat*}
	satisfying
	%
		\begin{align*}
			(\vect{v}_{s} - \Pi_{\vect{v}_{s},K} \vect{v}_{s} , \vect{w})_{K} &= 0 & & \forall \vect{w} \in {P}_{k-1}(K; \mathbb{V}), 
			\\
			-(\divergence (\sigma - \Pi_{\sigma,K} \sigma), \vect{w}_{s})_{K} &
            \\
            \notag 
            + \langle \tau_{s} {P_{\widehat{V}_s}} (\vect{v}_{s} - \Pi_{\vect{v}_{s},K} \vect{v}_{s}), \vect{w}_{s}\rangle_{\partial K} &= \langle R_K(\sigma, \vect{v}_{s}), \vect{w}_{s} \rangle_{\partial K} & & \forall \vect{w}_{s} \in {P}_{k+1}(K; \mathbb{V}), 
			\\
			-\langle (\sigma - \Pi_{\sigma,K} \sigma ) \vect{n} - \tau_{s}(\vect{v}_{s} - \Pi_{\vect{v}_{s},K}\vect{v}_{s} ), \widehat{\vect{w}}_{s} \rangle_{\partial K} &= \langle R_K(\sigma, \vect{v}_{s}), \widehat{\vect{w}}_{s} \rangle_{\partial K}  & & \forall \widehat{\vect{w}}_{s} \in \mathcal{R}_{k}(\partial K; \mathbb{V})
		\end{align*}
	%
	and if $\sigma \in H^m(K; \mathbb{S})$, $\vect{v}_{s} \in H^{m+1}(K; \mathbb{V})$ for $r\le m \le k+1$, then 
	\begin{multline}	    
		\label{eq:elasticity-interpolation-estimate}
		\| \sigma - \Pi_{\sigma,K} \sigma\|_{K} + h_K^{-1}\| \vect{v}_{s} - \Pi_{\vect{v}_{s},K} \vect{v}_{s} \|_{K} + h_K^{1/2} \| R_K(\sigma, \vect{v}_s) \|_{\partial K}
        \\
        \le Ch_K^m (|\sigma|_{m, K} + |\vect{v}_{s}|_{m+1, K}) .
	\end{multline}
\end{theorem}
\begin{proof}
	See \cite[Theorem~2.1]{Du-Sayas:2020a} and \cite[Theorem~3.1]{Du-Sayas:2021}. 
\end{proof}
\begin{theorem}
	Let $K$ be a simplex,  $k\geq 0$, $r > \frac 12$, and $\tau^{\max}_{f,K} = \max \tau_f|_{\partial K} > 0$. Then, there exists a projection 
	\begin{align*}
		\Pi_{\vect{v}_{f},K} \times \Pi_{p,K} : H^r(K; \mathbb{V}) \times H^r(K) \to {P}_k(K; \mathbb{V}) \times {P}_{k}(K)
	\end{align*}
	defined by 
	%
		\begin{alignat*}{4}
			(\vect{v}_{f} - \Pi_{\vect{v}_{f},K} \vect{v}_{f} , \vect{w})_{K} &= 0 & & \quad\forall \vect{w} \in {P}_{k-1}(K; \mathbb{V}), 
			\\
			(p - \Pi_{p,K} p , q)_{K} &= 0 & & \quad \forall q \in {P}_{k-1}(K), 
			\\
			\langle (\vect{v}_{f} - \Pi_{\vect{v}_{f},K} \vect{v}_{f}) \cdot \vect{n} + \tau_{f}(p - \Pi_{p,K}p ), \widehat{\vect{w}}_{s} \rangle_{\partial K} &= 0  & &\quad  \forall \widehat{\vect{w}}_{s} \in \mathcal{R}_{k}(\partial K)
		\end{alignat*}
	%
	and if $\vect{v}_{f} \in H^m(K; \mathbb{V})$, $p \in H^{m+1}(K)$ for $r\le m \le k+1$, then 
	\begin{align}
		\label{eq:darcy-interpolation-estimate}
		\| \vect{v}_{f} - \Pi_{\vect{v}_{f},K} \vect{v}_{f}\|_{K} + \| p - \Pi_{p,K} p \|_{K} 
		\le Ch_K^m (|\vect{v}_{f}|_{m, K} + |p|_{m, K}) .
	\end{align}
\end{theorem}
\begin{proof}
	See \cite{Cockburn:2010} and \cite{Du:book}. 
\end{proof}
The global projection operators are then defined by their restriction to each element $K$ by $\Pi_{\sigma}$, by $\Pi_{\sigma} \omega|_K := \Pi_{\sigma,K} \omega$ for $K\in \mathcal{T}_h$. Similarly we define $\Pi_{\vect{v}_s}$, $\Pi_{\vect{v}_f}$, $\Pi_{p}$, and $R(\sigma,\vect{v}_s)$.

\subsection{Error equations}\label{subsec:error-equations}

Let $(\sigma_{h}, \vect{v}_{s,h}, \vect{v}_{f,h}, p_{h}, \widehat{\vect{v}}_{s,h}, \widehat{p}_{h})$ be the HDG approximate solution of \eqref{eq:semidiscrete-eqs} with initial numerical data satisfying the approximation properties \eqref{eq:initial-data-constraint1}, \eqref{eq:initial-data-constraint2}, and the compatibility conditions \eqref{eq:initial-data-approximation1}, \eqref{eq:initial-data-approximation2}. Let $(\sigma, \vect{v}_s, \vect{v}_f, p)$ and $\widehat{\vect{v}}_s = \vect{v}_s|_{\Faces}, \widehat{p} = p|_{\Faces}$ the exact solution of \eqref{eq:1stordersystem}. Let $(\Pi_{\sigma}, \Pi_{\vect{v}_{s}})$ and $(\Pi_{\vect{v}_{f}}, \Pi_{p})$ be HDG projection operators onto the approximation spaces $\Sigma_h\times V_{s,h}$ and $V_{f,h}\times Q_h$, respectively,  defined in the previous section. We also denote by $P_{\widehat{V}_{s}}$ and $P_{\widehat{Q}}$ the $L^2$-projections onto the spaces $\widehat{V}_{s,h}$ and $\widehat{Q}_h$, respectively. Then we define the error of the approximate solutions $e_{\star}$, the intepolation error $e^I_{\star}$ and the projected error $e^{h}_{\star}$, for $\star = \sigma, \vect{v}_s, \vect{v}_f, p, \widehat{\vect{v}}_s, \widehat{p}$, by 
\begin{subequations}
	\label{eq:semidiscrete-error-decomposition-eqs}
	\begin{alignat}{4}
		\label{eq:semidiscrete-error-decomposition-eq1}
		e_{\sigma} &:= \sigma - \sigma_h & &= (\sigma - \Pi_{\sigma} \sigma) - (\sigma_h - \Pi_{\sigma} \sigma) & &=: e_{\sigma}^I - e_{\sigma}^h , 
		\\
		\label{eq:semidiscrete-error-decomposition-eq2}
		e_{\vect{v}_s} &:= \vect{v}_{s} - \vect{v}_{s,h} & &= (\vect{v}_{s} - \Pi_{\vect{v}_{s}} \vect{v}_{s}) - (\vect{v}_{s,h} - \Pi_{\vect{v}_{s}} \vect{v}_{s}) & &=: e_{\vect{v}_{s}}^I - e_{\vect{v}_{s}}^h , 
		\\
		\label{eq:semidiscrete-error-decomposition-eq3}
		e_{\vect{v}_f} &:= \vect{v}_{f} - \vect{v}_{f,h} & &= (\vect{v}_{f} - \Pi_{\vect{v}_{f}} \vect{v}_{f}) - (\vect{v}_{f,h} - \Pi_{\vect{v}_{f}} \vect{v}_{f}) & &=: e_{\vect{v}_{f}}^I - e_{\vect{v}_{f}}^h , 
		\\
		\label{eq:semidiscrete-error-decomposition-eq4}
		e_{p} &:= p - p_{h} & &= (p - \Pi_{p} p) - (p_{h} - \Pi_{p} p) & &=: e_{p}^I - e_{p}^h , 
		\\
		\label{eq:semidiscrete-error-decomposition-eq5}
		e_{\widehat{\vect{v}}_{s}} &= \widehat{\vect{v}}_{s} - \widehat{\vect{v}}_{s,h} & &= (\widehat{\vect{v}}_{s} - P_{\widehat{V}_{s}} \widehat{\vect{v}}_{s}) - (\widehat{\vect{v}}_{s,h} - P_{\widehat{V}_{s}} \widehat{\vect{v}}_{s}) & &=: \widehat{e}_{\vect{v}_{s}}^I - \widehat{e}_{\vect{v}_{s}}^h , 
		\\
		\label{eq:semidiscrete-error-decomposition-eq6}
		e_{\widehat{p}} &:= \widehat{p} - \widehat{p}_{h} & &= (\widehat{p} - {P_{\widehat{Q}}} \widehat{p}) - (\widehat{p}_{h} - {P_{\widehat{Q}}} \widehat{p}) & &=: \widehat{e}_{p}^I - \widehat{e}_{p}^h .
	\end{alignat}
\end{subequations}

Next, we derive an identity of the projected errors.
\begin{lemma}\label{eq:interpolationerror_eqns}
	Consider the definitions of the projected errors above $(e_\sigma^h, e_{\vect{v}_s}^{h}, e_{\vect{v}_f}^{h}, e_{p}^{h}, \widehat{e}_{\vect{v}_s}^{h}, \widehat{e}_{p}^{h})$. Then, they satisfy the following system
	\begin{subequations}
		\label{eq:semidiscrete-reduced-error-eqs}
		\begin{align}
			\label{eq:semidiscrete-reduced-error-eq1}
			\Thprod{\Ael(\dot{e}_{\sigma}^{h} + \alpha \dot{e}_{p}^{h} \Id)}{r} +\Thprod{\divergence r}{e_{\vect{v}_{s}}^{h}} -\pThprod{\widehat{e}_{\vect{v}_{s}}^{h}}{r \vect{n}} 
			&=  F_1^{I}(r),
			\\ 
			\label{eq:semidiscrete-reduced-error-eq25}			\Thprod{\rho_{11}\dot{e}_{\vect{v}_{s}}^{h}+\rho_{12}\dot{e}_{\vect{v}_{f}}^{h}}{\vect{w}_{s}} + \Thprod{e_{\sigma}^{h}}{\grad \vect{w}_{s}} &
            \\
            \notag 
            - \pThprod{e_{\sigma}^{h} \vect{n} - \tau_{s} ({P_{\widehat{V}_s}}{e}_{\vect{v}_{s}}^{h} - \widehat{e}_{\vect{v}_{s}}^{h})}{\vect{w}_{s}} 
			&=F^{I}_2(\vect{w}_s), 
			\\
			\label{eq:semidiscrete-reduced-error-eq3}			\Thprod{\rho_{12}\dot{e}_{\vect{v}_{s}}^{h}+\rho_{22}\dot{e}_{\vect{v}_{f}}^{h}}{\vect{w}_{f}} + \Thprod{\frac{\eta}{\kappa} e_{\vect{v}_{f}}^{h}}{\vect{w}_{f}} &
            \\
            \notag 
            - \Thprod{e_{p}^{h} }{\divergence\vect{w}_{f}} +\pThprod{\widehat{e}_{p}^{h}}{\vect{w}_{f}\cdot\vect{n}} 
			& = F_{3}^{I}(\vect{w}_{f}), 
			\\
			\label{eq:semidiscrete-reduced-error-eq4}
			\Thprod{s_{0}\dot{e}_{p}^{h}}{q} + \Thprod{\Ael(\dot{e}_{\sigma}^{h}+\alpha \dot{e}_{p}^{h}\Id)}{\alpha q\Id} - \Thprod{e_{\vect{v}_{f}}^{h}}{\grad q} &
            \\
            \notag 
            + \pThprod{e_{\vect{v}_{f}}^{h}\cdot \vect{n} + \tau_f (e_p^{h} - \hat{e}_p^{h})}{q} 
			&  = F_{4}^{I}(q),
			\\
			\label{eq:semidiscrete-reduced-error-eq5}
			\pThprod{e_{\sigma}^{h} \vect{n} + \tau_{s}(e_{\vect{v}_s}^{h} - \widehat{e}_{\vect{v}_s}^{h})}{\widehat{\vect{w}}_s}&= F_5^{I}(\widehat{\vect{w}}_s),
			\\
			\label{eq:semidiscrete-reduced-error-eq6}
			\pThprod{e_{\vect{v}_{f}}^{h}\cdot \vect{n} + \tau_{f}(e_{p}^{h} - \widehat{e}_{p}^{h})}{\widehat{q}}& = 0 
		\end{align}
		for any $(r, \vect{w}_s, \vect{w}_f, q, \widehat{\vect{w}}_s, \widehat{q}) \in \Sigma_{h} \times V_{s,h} \times V_{f,h} \times Q_{h} \times \widehat{V}_{s,h}\times \widehat{Q}_{h}$, and with functionals depending on the projection errors $e_\cdot^I$, given by
		\begin{align*}
			F_1^{I}(r) &= \Thprod{\Ael(\dot{e}_\sigma^{I} + \alpha \dot{e}_p^{I} \Id )}{r}, 
			\\
			F_2^{I}(\vect{w}_s) &= \Thprod{\rho_{11}\dot{e}_{\vect{v}_{s}}^{I}+\rho_{12}\dot{e}_{\vect{v}_{f}}^{I}}{\vect{w}_{s}} + \pThprod{R(\sigma, \vect{v}_s)}{\vect{w}_s},\\
			F_3^{I}(\vect{w}_f) &= \Thprod{\rho_{12}\dot{e}_{\vect{v}_s}^{I} + \rho_{22}\dot{e}_{v_f}^{I}} {\vect{w}_f} + \Thprod{\frac{\eta}{\kappa} e_{\vect{v}_f}^{I}}{\vect{w}_f}, \\ 
			F_4^{I}(q) &=\Thprod{s_0 \dot{e}_p^{I}}{q} + \Thprod{\Ael(\dot{e}_\sigma^I + \alpha \dot{e}_p^{I} \Id)}{\alpha q\Id},\quad F_5^{I}(\widehat{\vect{w}}_s) = \pThprod{R(\sigma, \vect{v}_s)}{\widehat{\vect{w}}_s}.
		\end{align*}
	\end{subequations}
\end{lemma}
\begin{proof}
	We begin by deriving the error equations. By subtracting the equation of the semidiscrete system \eqref{eq:semidiscrete-eqs} from those satisfied by the exact solution \eqref{eq:variational-eqs}, we obtain
	\begin{subequations}
		\label{eq:semidiscrete-error-eqs}
		\begin{align}
			\label{eq:semidiscrete-error-eq1}
			\Thprod{\Ael(\dot{e}_{\sigma} + \alpha \dot{e}_{p} \Id)}{r} +\Thprod{\divergence r}{e_{\vect{v}_{s}}} -\pThprod{\widehat{e}_{\vect{v}_{s}}}{r \vect{n}} &= 0 ,
			\\ 
			\label{eq:semidiscrete-error-eq2}
			\Thprod{\rho_{11}\dot{e}_{\vect{v}_{s}} +\rho_{12}\dot{e}_{\vect{v}_{f}}}{\vect{w}_{s}} + \Thprod{e_{\sigma}}{\grad \vect{w}_{s}} &
            \\
            \notag 
            - \pThprod{e_{\sigma} \vect{n}- \tau_{s} (\vect{v}_{s} - {P_{\widehat{V}_s}} \vect{v}_{s,h} - \widehat{e}_{\vect{v}_{s}})}{\vect{w}_{s}} & = 0, 
			\\
			\label{eq:semidiscrete-error-eq3}
			\Thprod{\rho_{12}\dot{e}_{\vect{v}_{s}}+\rho_{22}\dot{e}_{\vect{v}_{f}}}{\vect{w}_{f}} + \Thprod{\frac{\eta}{\kappa} e_{\vect{v}_{f}}}{\vect{w}_{f}} - \Thprod{e_{p} }{\divergence\vect{w}_{f}} +\pThprod{\widehat{e}_{p}}{\vect{w}_{f}\cdot\vect{n}} & = 0, 
			\\
			\label{eq:semidiscrete-error-eq4}
			\Thprod{s_{0}\dot{e}_{p}}{q} + \Thprod{\Ael(\dot{e}_{\sigma}+\alpha \dot{e}_{p}\Id)}{\alpha q\Id} - \Thprod{e_{\vect{v}_{f}}}{\grad q} &
            \\
            \notag 
            + \pThprod{e_{\vect{v}_{f}}\cdot \vect{n} + \tau_f (e_p - \widehat{e}_p)}{q} & = 0,
			\\
			\label{eq:semidiscrete-error-eq5}
			\pThprod{e_{\sigma} \vect{n} - \tau_{s} (P_{\widehat{V}_s} e_{\vect{v}_{s}} - \widehat{e}_{\vect{v}_{s}})}{\widehat{\vect{w}}_{s}} & = 0 ,
			\\
			\label{eq:semidiscrete-error-eq6}
			\pThprod{e_{\vect{v}_f}\cdot \vect{n} + \tau_{f}(e_{p} - \widehat{e}_{p})}{\widehat{q}} & = 0 
		\end{align}
	\end{subequations}
	for all $0<t<T$ and for any $(r, \vect{w}_{s}, \vect{w}_{f}, q, \widehat{\vect{w}}_{s}, \widehat{q}) \in \Sigma_{h} \times V_{h} \times V_{h} \times Q_{h} \times \widehat{V}_{f,h}\times \widehat{Q}_{h}$. We proceed by using the splitting of the error $e_{\star} = e_{\star}^{I} - e_{\star}^{h}$ in the error equations above. Then, the first equation gives
	\[
	\Thprod{\Ael(\dot{e}_{\sigma}^{I} + \alpha \dot{e}_{p}^{I} \Id)}{r} +\Thprod{\divergence r}{e_{\vect{v}_{s}}^{I}} -\pThprod{\widehat{e}_{\vect{v}_{s}}^{I}}{r \vect{n}} = \Thprod{\Ael(\dot{e}_{\sigma}^{I} + \alpha \dot{e}_{p}^{I} \Id)}{r} = F_{1}^{I}(r)
	\]
	where we use the definition of the operator $\Pi_{\vect{v}_s}$ and the definition of the $L^{2}-$projection $P_{\widehat{V}_s}$.
	The second equation follows after integration by parts and the definition of the operators $\Pi_{\sigma}, \Pi_{\vect{v}_s}, R$, and observing that $\pThprod{\tau_{s}\widehat{e}_{\vect{v}_s}^{I}}{\vect{w}_s} = 0$
	\begin{align*}
		\Thprod{\rho_{11}\dot{e}_{\vect{v}_{s}}^{I}+\rho_{12}\dot{e}_{\vect{v}_{f}}^{I}}{\vect{w}_{s}} + &\Thprod{e_{\sigma}^{I}}{\grad \vect{w}_{s}} - \pThprod{e_{\sigma}^{I} \vect{n} - \tau_{s} ({P_{\widehat{V}_s}}{e}_{\vect{v}_{s}}^{h} - \widehat{e}_{\vect{v}_{s}}^{I})}{\vect{w}_{s}} \\
		&=\Thprod{\rho_{11}\dot{e}_{\vect{v}_{s}}^{I}+\rho_{12}\dot{e}_{\vect{v}_{f}}^{I}}{\vect{w}_{s}} + 
		\pThprod{R(\sigma, \vect{v}_s)}{\vect{w}_s} = F^{I}_2(\vect{w}_s).
	\end{align*}
	Next, the definition of the operator $\Pi_{\vect{v}_f}$ it implies that $\Thprod{e_{p}^{I}}{\divergence \vect{w}_f} = 0$, and the $L^2-$projection $P_{\widehat{Q}}$ implies that  $\pThprod{\widehat{e}_{p}^{h}}{\vect{w}_{f}\cdot\vect{n}}=0 $. Then, it follows that
	\begin{align*}
		\Thprod{\rho_{12}\dot{e}_{\vect{v}_{s}}^{h}+\rho_{22}\dot{e}_{\vect{v}_{f}}^{I}}{\vect{w}_{f}} + &\Thprod{\frac{\eta}{\kappa} e_{\vect{v}_{f}}^{I}}{\vect{w}_{f}} - \Thprod{e_{p}^{h} }{\divergence\vect{w}_{f}} +\pThprod{\widehat{e}_{p}^{h}}{\vect{w}_{f}\cdot\vect{n}}  \\ 
		&= \Thprod{\rho_{12}\dot{e}_{\vect{v}_{f}}^{I}+\rho_{22}\dot{e}_{\vect{v}_{f}}^{I}}{\vect{w}_{f}} + \Thprod{\frac{\eta}{\kappa} e_{\vect{v}_{f}}^{I}}{\vect{w}_{f}} = F_{3}^{I}(\vect{w}_{f}).
	\end{align*}
	Similarly, using the definition of the HDG operator $(\Pi_{\vect{v}_{f}}, \Pi_{p})$, gives that
	\begin{multline*}
		\Thprod{s_{0}\dot{e}_{p}^{I}}{q} + \Thprod{\Ael(\dot{e}_{\sigma}^{I}+\alpha \dot{e}_{p}^{I}\Id)}{\alpha q\Id} - \Thprod{e_{\vect{v}_{f}}^{I}}{\grad q} + \pThprod{e_{\vect{v}_{f}}^{I}\cdot \vect{n} + \tau_f (e_p^{I} - \hat{e}_p^{I})}{q} 
        \\ 
		=  \Thprod{s_{0}\dot{e}_{p}^{I}}{q} + \Thprod{\Ael(\dot{e}_{\sigma}^{I}+\alpha \dot{e}_{p}^{I}\Id)}{\alpha q\Id}   = F_{4}^{I}(q).
	\end{multline*}
	The last two equations follow applying the definition of the operators $(\Pi_{\vect{v}_s}, \Pi_{\sigma}, R)$, and $P_{\widehat{V}_h}$ and $(\Pi_{\vect{v}_f}, \Pi_{p})$
	\begin{align*}
		\pThprod{e_{\sigma}^{I} \vect{n} + \tau_{s}(e_{\vect{v}_s}^{I} - \widehat{e}_{\vect{v}_s}^{I})}{\widehat{\vect{w}}_s}&=
		\pThprod{R(\sigma, \vect{v}_s) }{\widehat{\vect{w}}_s} - \pThprod{\tau_{s} \widehat{e}_{\vect{v}_s}^{I}}{\widehat{\vect{w}}_s} 
        \\
        &= \pThprod{R(\sigma, \vect{v}_s) }{\widehat{\vect{w}}_s} =F_5^{I}(\widehat{\vect{w}}_s),
		\\
		\pThprod{e_{\vect{v}_{f}}^{I}\cdot \vect{n} + \tau_{f}(e_{p}^{I} - \widehat{e}_{p}^{I})}{\widehat{q}}& = 0 .
	\end{align*}
\end{proof}

\subsection{Semidiscrete error estimates} \label{subsec:energy-estimates} 
We conclude this section by presenting its main result, the semidiscrete error estimate of the HDG approximation given in \eqref{eq:semidiscrete-eqs}.
\begin{theorem}\label{thm:semidiscrete-error-estimate}
	Suppose that $(\sigma_{h}, \vect{v}_{s,h}, \vect{v}_{f,h}, p_{h}, \widehat{\vect{v}}_{s,h}, \widehat{p}_{h})$ is a solution of \eqref{eq:semidiscrete-eqs} with numerical initial condition satisfying \eqref{eq:initial-data-constraint1}, \eqref{eq:initial-data-constraint2}, \eqref{eq:initial-data-approximation1}, \eqref{eq:initial-data-approximation2}. If exact solutions satisfy the regularity assumptions to make the quantities in the below estimates well-defined, then we have 
	\begin{align}
		\label{eq:semidiscrete-error-estimate}
		&\| e_{\sigma} (t) + \alpha e_{p}(t)\Id \|_{\Ael} + \|e_{p}(t)\|_{s_0} + \rho_0 (\|e_{\vect{v}_{s}}(t)\|_{L^2} + \|e_{\vect{v}_{f}}(t)\|_{L^2} )
		\\
		\notag
		&\leq C h^m (| \sigma(0), \vect{v}_{f}(0), p(0)|_{m} + |\vect{v}_{s}(0)|_{m})
		\\ 
		\notag
		&\quad + C h^m \left[ \int_0^t (| \dot{\sigma}(s), {\dot{\vect{v}}_{f}}(s), \dot{p}(s)|_{m} + |\dot{\vect{v}}_{s}(s) |_{m+1}) \,ds +  | \sigma (t), p(t), \vect{v}_{f}(t)|_{m} + |\vect{v}_{s}(t)|_{m} \right]
		\\ 
		\notag
		&\quad {
			+ C h^m \left[ \int_0^t (| {\sigma}(s), {\vect{v}}_{f}(s)|_{m}^2 + |{\vect{v}}_{s}(s) |_{m+1}^2) \,ds \right]^{\frac 12} }        
	\end{align}
	for $1\le m \le k+1$. 
	Moreover, if we assume that the coefficients satisfy an assumption 
	\begin{align}
		\Ael (x) \alpha \Id : \alpha \Id \le C_s s_0 (x) \qquad \text{a.e. } x \in \Omega    
	\end{align}
	with a uniform constant $C_s>0$, then $\| e_{\sigma} (t) \|_{\Ael}$ satisfies an error estimate with the same right-hand side of \eqref{eq:semidiscrete-error-estimate} and a constant depending on $C_s$. 
\end{theorem}
\begin{proof}   
	The proof follows applying the energy equality Lemma \ref{lemma:energy} for $(e_{\sigma}^{h}, e_{\vect{v}_{s}}^h, e_{\vect{v}_{f}}^h, e_{p}^h, \widehat{e}_{\vect{v}_{s}}^{h}, \widehat{e}_{p}^{h})$ and 
	%
	%
	\begin{align*}
		X(t)^2 &= \Thprod{\Ael(e_{\sigma}^{h}(t) + \alpha e_{p}^{h}(t) \Id)}{e_{\sigma}^{h} + \alpha e_{p}^{h}(t) \Id)} +\Thprod{s_{0}e_{p}^{h}(t)}{e_{p}^h(t)}
		\\
		&\quad + \Thprod{\rho_{11} e_{\vect{v}_{s}}^{h}(t)}{e_{\vect{v}_{s}}^{h}(t)} +\Thprod{2\rho_{12} e_{\vect{v}_{s}}^{h}(t)}{e_{\vect{v}_{f}}^{h}(t)} + \Thprod{\rho_{22} e_{\vect{v}_{f}}^{h}(t)}{e_{\vect{v}_{f}}^{h}(t)} ,
		\\
		Y(t)^2 &= \int_0^t \left( \Thprod{\frac{\eta}{\kappa} e_{\vect{v}_{f}}^{h}(s)}{e_{\vect{v}_{f}}^{h}(s)} + \pThprod{\tau_{s} (P_{\widehat{V}_s} e_{\vect{v}_{s}}^{h}(s) - \widehat{e}_{\vect{v}_{s}}^{h}(s))}{P_{\widehat{V}_s} e_{\vect{v}_{s}}^{h}(s) - \widehat{e}_{\vect{v}_{s}}^{h}(s)} \right) \,ds
		\\
		&\quad + \int_0^t \pThprod{\tau_{f} (e_{p}^{h}(s) - \widehat{e}_{p}^{h}(s))}{e_{p}^{h}(s) - \widehat{e}_{p}^{h}(s)} \,ds .
	\end{align*}
	Thus, by Lemma \ref{lemma:energy}, 
	\begin{multline}
		\label{eq:semidiscrete-energy-error}
		X(t)^2 - X(0)^2 + 2Y(t)^2 
        \\
        = \left( (F_1^I(\sigma_h(s)) + F_2^I(\widehat{\vect{v}}_{s,h}(s)) + F_3^I(\widehat{\vect{v}}_{f,h}) + F_4^I(p_h(s)) + F_5^I(\widehat{\vect{v}}_{s,h}(s))\right) \, ds
	\end{multline}
	where the last term comes from $\pThprod{R(\sigma, \vect{v}_{s})}{e_{\vect{v}_{s}}^{h}} = \pThprod{R(\sigma, \vect{v}_{s})}{P_{\widehat{V}_s} e_{\vect{v}_{s}}^{h}}$. 
	
	Applying spatial and space-time Cauchy--Schwarz inequalities to \eqref{eq:semidiscrete-energy-error} gives
	\begin{align*}
		X(t)^2 - X(0)^2 + 2 Y(t)^2 \le \int_0^t F(s) X(s)\,ds + G(t) Y(t) 
	\end{align*}
	where 
	\begin{align*}
		F(s) &= \|\dot{e}_{\sigma}^{I}(s) + \alpha \dot{e}_{p}^{I}(s) \Id\|_{\Ael} + \|\dot{e}_p^I\|_{s_0} + \rho_0^{-1/2} \| \rho_{11}^{1/2} \dot{e}_{\vect{v}_{s}}^I, \rho_{12}^{1/2} \dot{e}_{\vect{v}_{s}}^I, \rho_{12}^{1/2} \dot{e}_{\vect{v}_{f}}^I, \rho_{22}^{1/2} \dot{e}_{\vect{v}_{s}}^I\| , 
		\\
		G(t) &= \left[\int_0^t \left(\left\| \frac{\eta^{1/2}}{\kappa^{1/2}} e_{\vect{v}_{f}}^I (s) \right\|_{L^2}^2  + \tau_{s}^{-1}\| R(\sigma(s), \vect{v}_{s}(s)) \|_{\partial \mathcal{T}_h}^2 \right) \,ds\right]^{\frac 12} .
	\end{align*}
	We now show  
		\begin{align}
			\label{eq:integral-estimate}
			X(t) \le X(0) + \int_0^t F(s) \,ds + \frac{1}{2\sqrt{2}} G(t) .
		\end{align}
		A similar result with bigger constants was obtained in \cite[Lemma~1]{Lee:2016} but here we claim \eqref{eq:integral-estimate} which has smaller constants and give a self-contained proof. First, by Young's inequality one can get 
		\begin{align}
			\label{eq:integral-inequality}
			X(t)^2 - X(0)^2 \le \int_0^t F(s) X(s)\,ds + \frac 18 G(t)^2 .
		\end{align}
		Suppose that $X(t_M) = \max_{0\le s\le t} X(s)$ for $0\le t_M \le t$. Note that \eqref{eq:integral-inequality} also holds for $t_M$, so we have 
		\begin{align*}
			X(t_M)^2 - X(0)^2 \le \int_0^{t_M} F(s) X(s) \,ds + \frac 18 G(t_M)^2 \le \int_0^{t_M} F(s) \,ds X(t_M) + \frac 18 G(t_M)^2 .
		\end{align*}
		Regarding the completing square formula, 
		\begin{align*}
			\left( X(t_M) - \frac 12 \int_0^{t_M} F(s)\,ds \right)^2 \le X(0)^2 + \frac 14 \left(\int_0^{t_M} F(s)\,ds\right)^2 + \frac 18 G(t_M)^2 .
		\end{align*}
		Taking square roots of both sides and the inequality $\sqrt{a^2 + b^2 + c^2} \le a + b + c$ give 
		\begin{align*}
			X(t_M) \le X(0) + \int_0^{t_M} F(s)\,ds + \frac 1{2\sqrt{2}} G(t_M) .
		\end{align*}
		Finally, \eqref{eq:integral-estimate} follows because $X(t) \le X(t_M)$, and $\int_0^t F(s)\,ds$, $G(t)$ are non-decreasing in $t$. 	
	We remark that the projections $(\Pi_{\sigma}, \Pi_{\vect{v}_{s}})$ and $(\Pi_{\vect{v}_{f}}, \Pi_{p})$ are independent of the time variable $t$, so the time derivative commute with these projection operators. 
	From the approximation properties \eqref{eq:elasticity-interpolation-estimate} and \eqref{eq:darcy-interpolation-estimate}, 
	\begin{align*}
		F(s) &\le C_F h^m | \dot{\sigma}(s), {\dot{\vect{v}}_{f}}(s), \dot{p}(s), \dot{\vect{v}}_{s}(s) |_{m} ,
		\\
		G(t) &\le C_G h^m \left[\int_0^t | \vect{v}_{f}(s), \sigma(s), \vect{v}_{s}(s)|_{m}^2 \,ds \right]^{\frac 12}
	\end{align*}
	for $1 \le m \le k+1$ where $C_{F}$ has the factor $\max\{C_{\Ael}, \|s_0\|_{L^{\infty}}, \|\rho_0^{-1} \rho_{ij}\|_{L^{\infty}}, 1\le i, j \le 2 \}$ and $C_{G}$ depends on $\|\kappa^{-1} \eta \|_{L^{\infty}}$. By \eqref{eq:integral-estimate}, we can obtain
	\begin{align}
		\label{eq:semidiscrete-eh-error-estimate}
		&\| e_{\sigma}^{h} (t) + \alpha e_{p}^{h}(t)\Id \|_{\Ael} + \|e_{p}^{h}(t)\|_{s_0} + \rho_0 (\|e_{\vect{v}_{s}}^{h}(t)\|_{L^2} + \|e_{\vect{v}_{f}}^{h}(t)\|_{L^2} 
		\\
		\notag
		&\le \| e_{\sigma}^{h} (0) + \alpha e_{p}^{h}(0)\Id \|_{\Ael} + \|e_{p}^{h}(0)\|_{s_0} + \rho_0 (\|e_{\vect{v}_{s}}^{h}(0)\|_{L^2} + \|e_{\vect{v}_{f}}^{h}(0)\|_{L^2} )
		\\ 
		\notag
		&\quad + C_{F} h^m \int_0^t (| \dot{\sigma}(s), {\dot{\vect{v}}_{f}}(s), \dot{p}(s), \dot{\vect{v}}_{s}(s) |_{m}) \,ds + C_G h^m \left[\int_0^t | \vect{v}_{f}(s), \sigma(s), \vect{v}_{s}(s)|_{m}^2 \,ds \right]^{\frac 12} .
	\end{align}
	By the triangle inequality and \eqref{eq:initial-data-approximation1}, \eqref{eq:initial-data-approximation2}, 
	\begin{multline*}
		\| e_{\sigma}^{h} (0) + \alpha e_{p}^{h}(0)\Id \|_{\Ael} + \|e_{p}^{h}(0)\|_{s_0} + \rho_0 (\|e_{\vect{v}_{s}}^{h}(0)\|_{L^2} + \|e_{\vect{v}_{f}}^{h}(0)\|_{L^2} )
		\\
		\le C h^m (| \sigma(0), \vect{v}_{f}(0), p(0), \vect{v}_{s}(0)|_{m})  .
	\end{multline*}
	Applying this to  \eqref{eq:semidiscrete-eh-error-estimate}, the triangle inequality and \eqref{eq:darcy-interpolation-estimate}, \eqref{eq:elasticity-interpolation-estimate} give \eqref{eq:semidiscrete-error-estimate}, then 
	\begin{align*}
		\| e_{\sigma} (t) \|_{\Ael} \le \| e_{\sigma} (t) + \alpha e_{p}(t)\Id \|_{\Ael} + \| \alpha e_{p}(t)\Id \|_{\Ael} \le \| e_{\sigma} (t) + \alpha e_{p}(t)\Id \|_{\Ael} + C_s\| e_{p}(t) \|_{s_0},
	\end{align*}
	so an estimate of $\| e_{\sigma} (t) \|_{\Ael}$ with the same bound in \eqref{eq:semidiscrete-eh-error-estimate} up to $C_s$. 
\end{proof}

\section{Fully Discrete Error Analysis}
\label{sec:fully-discrete-error-analysis}

{
In this section, we extend our analysis to the fully discrete setting by applying the Crank--Nicolson time-stepping scheme to the semidiscrete HDG formulation. We begin in Section \ref{ss:fully_discrete_scheme} by introducing the necessary temporal notation, detailing the fully discrete scheme, and establishing its well-posedness. Next, Section \ref{ss:fully_discrete_result} focuses on systematically deriving the fully discrete error equations. Finally, by employing an auxiliary discrete inequality, we present in Theorem \ref{thm:discrete-error-estimates} the main theoretical result of this paper: the optimal \textit{a priori} error estimates for the fully discrete approximation.
}

{
Before detailing the fully discrete scheme, it is important to justify the choice of the time integration method. Poroelastic wave equations are known to be a highly stiff system, primarily due to the strong solid-fluid friction term and the large contrast in wave speeds, which induces the highly dissipative nature of Biot's slow P-wave. 
}

{
If an explicit time-stepping scheme were employed, the Courant-Friedrichs-Lewy (CFL) stability condition would impose severe and prohibitively small time step restrictions, completely dictated by the stiffest component of the system rather than accuracy requirements. To overcome this computational bottleneck, we employ the Crank-Nicolson method. As an implicit scheme, it provides unconditional stability, allowing the time step size $\Delta t$ to be chosen based solely on the desired accuracy to resolve the wavefield. Furthermore, the Crank-Nicolson scheme is second-order accurate in time and non-dissipative, meaning it preserves the discrete energy of the system without introducing spurious numerical damping, a crucial property for accurate long-term simulations of wave propagation.
}
\subsection{Fully discrete scheme}\label{ss:fully_discrete_scheme}
Let $\Delta t>0$ be the time step size and $N$ be a positive integer. We define the intermediate time steps $t_i = i \Delta t$ for a nonnegative integer $0\le i\le N$. If a function $f$ is defined for all $t\in[0,N\Delta t]$, then we denote $f^i:= f(t_i)$ for $t_i \le N\Delta t$. In the same way, $\sigma^{i}, \vect{v}_{s}^{i}, \widehat{\vect{v}}_{s}^{i}, \vect{v}_{f}^{i}, p^{i}, \widehat{p}^{i}$ are $\sigma(t_i), \vect{v}_{s}(t_i), \widehat{\vect{v}}_{s}(t_i), \vect{v}_{f}(t_i), p(t_i), \widehat{p}(t_i)$. In addition, $\dot{\xi}^i$ is $\dot{\xi}(t_i)$ for $\xi = \sigma, \vect{v}_{s}, \widehat{\vect{v}}_{s}, \vect{v}_{f}, p, \widehat{p}$. 
{However, $\sigma_h^{i}, \vect{v}_{s,h}^{i}, \widehat{\vect{v}}_{s,h}^{i}, \vect{v}_{f,h}^{i}, p_h^{i}, \widehat{p}_h^{i}$ are {\it not} the evaluation of the semidiscrete solutions in \eqref{eq:semidiscrete-eqs} at $t=t_i$ but the fully discrete solution of the variables for the unknowns $\sigma, \vect{v}_{s}, \widehat{\vect{v}}_{s}, \vect{v}_{f}, p, \widehat{p}$ at the $i$-th time step which will be defined later in our fully discrete scheme.}
Once a sequence $\{g^i\}_{i=0}^{N}$ for a positive integer $N$ is defined, we define 
\begin{align}
	\label{time-quotient}
	d_t g^{i+1} := \frac{g^{i+1} - g^{i}}{\Delta t} \quad \text{and} \quad {g}^{i+\frac 12 } := \frac{g^i + g^{i+1}} 2 \qquad \text{for } 0 \le i \le N-1.
\end{align}
%


Given $(\sigma_h^i, \vect{v}_{s,h}^i, \widehat{\vect{v}}_{s,h}^i, \vect{v}_{f,h}^i, p_h^i, \widehat{p}_h^i)$ and $\vect{f}^i$, $\vect{f}^{i+1}$, $g^i$, $g^{i+1}$, the Crank--Nicolson scheme seeks 
\begin{align*}
	(\sigma_h^{i+1}, \vect{v}_{s,h}^{i+1}, \widehat{\vect{v}}_{s,h}^{i+1}, \vect{v}_{f,h}^{i+1}, p_h^{i+1}, \widehat{p}_h^{i+1}) \in \Sigma_{h} \times V_{s,h} \times V_{f,h} \times Q_{h} \times \widehat{V}_{s,h}\times \widehat{Q}_{h}
\end{align*}
satisfying the following equations:  
\begin{subequations}
	\label{eq:fully-discrete-cn-eqs}
	\begin{align}
		\label{eq:fully-discrete-cn-eq1}
		\Thprod{\Ael(d_t \sigma_{h}^{i+1} + \alpha d_t{p}_{h}^{i+1} \Id)}{r} +\Thprod{\divergence r}{\vect{v}_{s,h}^{i+\frac 12}} -\pThprod{\widehat{\vect{v}}_{s,h}^{i+\frac 12}}{r \vect{n}} &= 0 ,
		\\ 
		\label{eq:fully-discrete-cn-eq2}		\Thprod{\rho_{11}d_t{\vect{v}}_{s,h}^{i+1}+\rho_{12}d_t{\vect{v}}_{f,h}^{i+1}}{\vect{w}_{s}} + \Thprod{\sigma_{h}^{i+\frac 12}}{\grad \vect{w}_{s}} &
        \\
        \notag 
        -\pThprod{\widehat{\sigma_{h} \vect{n}}^{i+\frac 12} }{\vect{w}_{s}} 
		&= {\Thprod{\vect{f}^{i+1/2}}{\vect{w}_{s}}}, 
		\\
		\label{eq:fully-discrete-cn-eq3}
		\Thprod{\rho_{12}d_t{\vect{v}}_{s,h}^{i+1}+\rho_{22}d_t{\vect{v}}_{f,h}^{i+1}}{\vect{w}_{f}} + \Thprod{\frac{\eta}{\kappa} \vect{v}_{f,h}^{i+\frac 12}}{\vect{w}_{f}} 
		&
        \\
        \notag 
        - \Thprod{p_{h}^{i+\frac 12} }{\divergence\vect{w}_{f}} +\pThprod{\widehat{p}_{h}^{i+\frac 12}}{\vect{w}_{f}\cdot\vect{n}} &= 0, 
		\\
		\label{eq:fully-discrete-cn-eq4}
		\Thprod{s_{0}d_t{p}_{h}^{i+1}}{q} + \Thprod{\Ael(d_t{\sigma}_{h}^{i+1}+\alpha d_t{p}_{h}^{i+1}\Id)}{\alpha q\Id} &
        \\
        \notag 
        - \Thprod{\vect{v}_{f,h}^{i+\frac 12}}{\grad q} 
		+\pThprod{\widehat{\vect{v}_{f,h}\cdot \vect{n}}^{i+\frac 12} }{q}
		&={\Thprod{g^{i+1/2} }{q}}, 
		\\
		\label{eq:fully-discrete-cn-eq5}
		\pThprod{\widehat{\sigma_h\vect{n}}^{i+\frac 12}  }{\widehat{\vect{w}}_{s}}
		=
		\pThprod{\sigma_h^{i+\frac 12} \vect{n} - \tau_{s} ({P_{\widehat{V}_s}} \vect{v}_{s,h}^{i+\frac 12} - \widehat{\vect{v}}_{s,h}^{i+\frac 12})}{\widehat{\vect{w}}_{s}} &= 0, 
		\\
		\label{eq:fully-discrete-cn-eq6}
		\pThprod{\widehat{\vect{v}_{f,h} \cdot \vect{n}}^{i+\frac 12} }{\widehat{q}}
		=
		\pThprod{\vect{v}_{f,h}^{i+\frac 12}\cdot \vect{n} + \tau_f (p_h^{i+\frac 12} - \widehat{p}_h^{i+\frac 12})}{\widehat{q}} &= 0
	\end{align}
\end{subequations}
for any $(r, \vect{w}_s, \vect{w}_f, q, \widehat{\vect{w}}_s, \widehat{q}) \in \Sigma_{h} \times V_{s,h} \times V_{f,h} \times Q_{h} \times \widehat{V}_{s,h}\times \widehat{Q}_{h}$. In the following lemma, we prove the well-posedness of the fully discrete scheme.
\begin{lemma}
	The system \eqref{eq:fully-discrete-cn-eqs} is well-posed.
\end{lemma}
\begin{proof}
	We first note that \eqref{eq:fully-discrete-cn-eqs} is a square linear system; thus, the uniqueness of solutions implies the existence of solutions. 
	Suppose then that $(\sigma_h^i, \vect{v}_{s,h}^i, \widehat{\vect{v}}_{s,h}^i, \vect{v}_{f,h}^i, p_h^i, \widehat{p}_h^i)$ and $g^i$, $g^{i+1}$, $\vect{f}^{i}$, $\vect{f}^{i+1}$ vanish. Then, \eqref{eq:fully-discrete-cn-eqs} reduces to 
	%
		\begin{align*}
			\Thprod{\Ael(\sigma_{h}^{i+1} + \alpha {p}_{h}^{i+1} \Id)}{r} +\frac{\Delta t}{2} (\Thprod{\divergence r}{\vect{v}_{s,h}^{i+1}} -\pThprod{\widehat{\vect{v}}_{s,h}^{i+1}}{r \vect{n}})& = 0 ,
			\\ 
			\Thprod{\rho_{11}{\vect{v}}_{s,h}^{i+1} +\rho_{12}{\vect{v}}_{f,h}^{i+1}}{\vect{w}_{s}} + \frac{\Delta t}{2}\left(\Thprod{\sigma_{h}^{i+1}}{\grad \vect{w}_{s}} 
			- \pThprod{\widehat{\sigma_{h} \vect{n}}^{i+1} }{\vect{w}_{s}}\right)& = 0, 
			\\
			\Thprod{\rho_{12}{\vect{v}}_{s,h}^{i+1}+\rho_{22}{\vect{v}}_{f,h}^{i+1}}{\vect{w}_{f}} + \frac{\Delta t}{2}\left(\Thprod{\frac{\eta}{\kappa} \vect{v}_{f,h}^{i+1}}{\vect{w}_{f}} 
			- \Thprod{p_{h}^{i+1} }{\divergence\vect{w}_{f}} +\pThprod{\widehat{p}_{h}^{i+1}}{\vect{w}_{f}\cdot\vect{n}}\right) &= 0, 
			\\
			\Thprod{s_{0}{p}_{h}^{i+1}}{q} + \Thprod{\Ael({\sigma}_{h}^{i+1}+\alpha {p}_{h}^{i+1}\Id)}{\alpha q\Id} + \frac{\Delta t}{2} \left(-\Thprod{\vect{v}_{f,h}^{i+1}}{\grad q} 
			+ \pThprod{\widehat{\vect{v}_{f,h}\cdot \vect{n}}^{i+1}}{q}\right) &= 0,
			\\
			\frac{\Delta t}{2}\pThprod{\widehat{\sigma_h \vect{n}}^{i+1} }{\widehat{\vect{w}}_{s}}=
			\frac{\Delta t}{2}\pThprod{\sigma_h^{i+1} \vect{n} - \tau_{s} (P_{\widehat{V}_s} \vect{v}_{s,h}^{i+1} - \widehat{\vect{v}}_{s,h}^{i+1})}{\widehat{\vect{w}}_{s}} &= 0, 
			\\
			\frac{\Delta t}{2}\pThprod{\widehat{\vect{v}_{f,h}\cdot \vect{n}}^{i+1})}{\widehat{q}}=
			\frac{\Delta t}{2}\pThprod{\vect{v}_{f,h}^{i+1}\cdot \vect{n} + \tau_f (p_h^{i+1} - \widehat{p}_h^{i+1})}{\widehat{q}} &= 0 .
		\end{align*}
	%
	By taking the test functions $r = \sigma_h^{i+1}$, $\vect{w}_{s} = \vect{v}_{s,h}^{i+1}$, $\vect{w}_{f} = \vect{v}_{f,h}^{i+1}$, $q = p_h^{i+1}$, $\widehat{\vect{w}}_{s} = \widehat{\vect{v}}_{s,h}^{i+1}$, $\widehat{q}=\widehat{p}_{h}^{i+1}$ in the system above and adding all the equations, we obtain
	\begin{align*}
		\Thprod{\Ael({\sigma}_{h}^{i+1} + \alpha {p}_{h}^{i+1} \Id)}{{\sigma}_{h}^{i+1} + \alpha {p}_{h}^{i+1} \Id)}
        \\
		+ \Thprod{\rho_{11}\vect{v}_{s,h}^{i+1} +\rho_{12}\vect{v}_{f,h}^{i+1}}{\vect{v}_{s,h}^{i+1}} +\Thprod{\rho_{12}\vect{v}_{s,h}^{i+1}+\rho_{22}\vect{v}_{f,h}^{i+1}}{\vect{v}_{f,h}^{i+1}} 
		\\
		+ \Thprod{s_{0}p_h^{i+1}}{p_h^{i+1}} + \frac{\Delta t}{2} \Thprod{\frac{\eta}{\kappa} \vect{v}_{f,h}^{i+1}}{\vect{v}_{f,h}^{i+1}} 
		\\
		+ \frac{\Delta t}{2}  \pThprod{\tau_{s} ({P_{\widehat{V}_s}} \vect{v}_{s,h}^{i+1} - \widehat{\vect{v}}_{s,h}^{i+1})}{P_{\widehat{V}_s}\vect{v}_{s,h}^{i+1}-\widehat{\vect{v}}_{s,h}^{i+1}}
		\\
		+ \frac{\Delta t}{2} \pThprod{\tau_{f} (p_{h}^{i+1} - \widehat{p}_{h}^{i+1})}{p_h^{i+1}-\widehat{p}_h^{i+1}}&=0 . 
	\end{align*}
	Then, using the positivity of $\Ael$ assumption \eqref{eq:A-positive},  and the assumption on the density coefficientes \eqref{eq:rho-coercivity}, we obtain that 
	$\vect{v}_{s,h}^{i+1}= \vect{v}_{f,h}^{i+1} =   0$, $ (\sigma_h^{i+1} + \alpha p_h^{i+1}) = 0, $ and 
	\begin{align*}
		\quad \pThprod{\tau_{s} ({P_{\widehat{V}_s}} \vect{v}_{s,h}^{i+1} - \widehat{\vect{v}}_{s,h}^{i+1})}{P_{\widehat{V}_s}\vect{v}_{s,h}^{i+1}-\widehat{\vect{v}}_{s,h}^{i+1}} = 0, \quad 
		\pThprod{\tau_{f} (p_{h}^{i+1} - \widehat{p}_{h}^{i+1})}{p_h^{i+1}-\widehat{p}_h^{i+1}} = 0. 
	\end{align*}
	Therefore, $\widehat{\vect{v}}_{s,h}^{i+1} = 0$ and $p_h^{i+1} = \widehat{p}_h^{i+1}$ on $\partial \mathcal{T}_h$. Moreover, replacing $\vect{v}_{s,h}^{i+1}= \vect{v}_{f,h}^{i+1} =   0$, and $p_h^{i+1} = \widehat{p}_h^{i+1}$ in the third equation of the system above, and using integration by parts gives $\grad p_h^{i+1} = 0$, so $p_h^{i+1}$ is constant on $\Omega$. Since $p_h^{i+1} = \widehat{p}_h^{i+1}$ on $\partial \mathcal{T}_h$ and $\widehat{p}_h^{i+1}|_{\Gamma_p} = 0$, $p_h^{i+1} = 0$ on $\Omega$. Then, $\sigma_h^{i+1} = 0$ and $\widehat{p}_h^{i+1} =0$. Therefore, \eqref{eq:fully-discrete-cn-eqs} is well-posed.
\end{proof}

%
\subsection{Main result: fully discrete error estimates}\label{ss:fully_discrete_result}
In this section, we introduce the main result, the error estimates of the fully discrete scheme \eqref{eq:fully-discrete-cn-eqs}. As in \eqref{eq:semidiscrete-error-decomposition-eqs} let 
%
\begin{align*}
	e_{\xi}^i &:= \xi^i - \xi_h^i & &= (\xi^i - \Pi_{\xi} \xi^i) - (\xi_h^i - \Pi_{\xi} \xi^i) & &=: e_{\xi}^{I,i} - e_{\xi}^{h,i}  \quad \text{for } \xi= \sigma, \vect{v}_s, \vect{v}_f, p,
	\\
	\widehat{e}_{\vect{v}_{s}}^i &= \widehat{\vect{v}}_{s}^i - \widehat{\vect{v}}_{s,h}^i & &= (\widehat{\vect{v}}_{s}^i - {P_{\widehat{V}_{s}}} \widehat{\vect{v}}_{s}^i) - (\widehat{\vect{v}}_{s,h}^i - {P_{\widehat{V}_{s}}} \widehat{\vect{v}}_{s}^i) & &=: \widehat{e}_{\vect{v}_{s}}^{I,i} - \widehat{e}_{\vect{v}_{s}}^{h,i} , 
	\\
	\widehat{e}_{p}^i &:= \widehat{p}^i - \widehat{p}_{h}^i & &= (\widehat{p}^i - {P_{\widehat{Q}}} \widehat{p}^i) - (\widehat{p}_{h}^i - {P_{\widehat{Q}}} \widehat{p}^i) & &=: \widehat{e}_{p}^{I,i} - \widehat{e}_{p}^{h,i} .
\end{align*}
%
%
\begin{lemma}[Fully discrete error equations] 
Consider the solution of the discrete scheme \eqref{eq:fully-discrete-cn-eqs}. Then, the projected errors $(e_{\sigma}^{h,i+1}, e_{\vect{v}_s}^{h,i+1}, e_{\vect{v}_f}^{h,i+1}, e_{p}^{h,i+1}, \widehat{e}_{\vect{v}_s}^{h,i+1}, \widehat{e}_{p}^{h,i+1})$ satisfy the following identity 
\begin{subequations}
	\label{eq:fully-discrete-reduced-error-eqs}
	\begin{align}
		\label{eq:fully-discrete-reduced-error-eq1}
		\Thprod{\Ael(d_t e_{\sigma}^{h,i+1} + \alpha d_te_{p}^{h,i+1} \Id)}{r} +\Thprod{\divergence r}{e_{\vect{v}_{s}}^{h,i+\frac 12}} &
        \\
        \notag 
        -\pThprod{\widehat{e}_{\vect{v}_{s}}^{h,i+\frac 12}}{r \vect{n}} 
		&= R_1^i(r) ,
		\\ 
		\label{eq:fully-discrete-reduced-error-eq2}
		\Thprod{\rho_{11}d_t e_{\vect{v}_{s}}^{h,i+1}+\rho_{12}d_t e_{\vect{v}_{f}}^{h,i+1}}{\vect{w}_{s}} + \Thprod{e_{\sigma}^{h,i+\frac 12}}{\grad \vect{w}_{s}} - \pThprod{e_{\sigma}^{h,i+\frac 12} \vect{n} }{\vect{w}_{s}}&
		\\
		\notag
		\qquad + \Thprod{e_{\sigma}^{I,i+\frac 12}}{\grad \vect{w}_s} - \pThprod{e_{\sigma}^{I,i+\frac 12} \vect{n} - \tau_s (\Pi_{\vect{v}_s} \vect{v}_{s}^{i+\frac 12} - {P_{\widehat{V}_s}}\widehat{\vect{v}}_{s}^{i+\frac 12})}{\vect{w}_s} &= {R_2^i(\vect{w}_s)}, 
		\\
		\label{eq:fully-discrete-reduced-error-eq3}
		\Thprod{\rho_{12}d_t e_{\vect{v}_{s}}^{h,i+1}+\rho_{22}d_t e_{\vect{v}_{f}}^{h,i+1}}{\vect{w}_f} + \Thprod{\frac{\eta}{\kappa} e_{\vect{v}_{f}}^{h,i+\frac 12}}{\vect{w}_{f}} - \Thprod{e_p^{h,i+\frac 12} }{\divergence\vect{w}_{f}} &
		\\
		\notag
		+\pThprod{\widehat{e}_p^{h,i+\frac 12}}{\vect{w}_{f}\cdot\vect{n}} &= R_3^i(\vect{v}_f) , 
		\\
		\label{eq:fully-discrete-reduced-error-eq4}
		\Thprod{s_{0}d_t e_{p}^{h,i+1}}{q} + \Thprod{\Ael(d_t e_{\sigma}^{h,i+1}+\alpha d_t e_{p}^{h,i+1}\Id)}{\alpha q\Id} - \Thprod{e_{\vect{v}_{f}}^{h,i+\frac 12}}{\grad q} &
		\\
		\notag
		\quad+ \pThprod{e_{\vect{v}_{f}}^{h,i+\frac 12}\cdot \vect{n} + \tau_{f} (e_p^{h,i+\frac 12} - \widehat{e}_p^{h,i+\frac 12})}{q}  
		&= R_4^i(q) ,
		\\
		\label{eq:fully-discrete-reduced-error-eq5}
		\pThprod{e_{\sigma}^{h,i+\frac 12} \vect{n} - \tau_{s} ({P_{\widehat{V}_s}} e_{\vect{v}_{s}}^{h,i+\frac 12} - \widehat{e}_{\vect{v}_{s}}^{h,i+\frac 12})}{\widehat{\vect{w}}_{s}} &
        \\
        \notag = \pThprod{e_{\sigma}^{I,i+\frac 12} \vect{n} - \tau_{s} ({P_{\widehat{V}_s}} e_{\vect{v}_{s}}^{I,i+\frac 12} - \widehat{e}_{\vect{v}_{s}}^{I,i+\frac 12})}{\widehat{\vect{w}}_{s}}&, 
		\\
		\label{eq:fully-discrete-reduced-error-eq6}
		\pThprod{e_{\vect{v}_{f}}^{h,i+\frac 12} + \tau_f (e_p^{h,i+\frac 12} - \widehat{e}_p^{h,i+\frac 12})}{\widehat{q}} &= 0 .
	\end{align}
\end{subequations}
for any 
$(r, \vect{w}_s, \vect{w}_f, q, \widehat{\vect{w}}_s, \widehat{q}) \in \Sigma_{h} \times V_{s,h} \times V_{f,h} \times Q_{h} \times \widehat{V}_{s,h}\times \widehat{Q}_{h}$ where
\begin{align*}
	R_1^{i}(r)  = & \Thprod{\Ael(d_t e_{\sigma}^{I,i+1} + \alpha d_te_{p}^{I,i+1} \Id)}{r} - \Thprod{\Ael(d_t \sigma^{i+1} - \dot{\sigma}^{i+\frac 12} + \alpha (d_t{p}^{i+1} - \dot{p}^{i+\frac 12})\Id)}{r},
	\\
	R_2^{i}(\vect{w}_s)  = &\Thprod{\rho_{11}d_t e_{\vect{v}_{s}}^{I,i+1}+\rho_{12}d_t e_{\vect{v}_{f}}^{I,i+1}}{\vect{w}_{s}} 
    \\
    & - \Thprod{\rho_{11}(d_t{\vect{v}}_{s}^{i+1}-\dot{\vect{v}}_s^{i+\frac 12}) +\rho_{12}(d_t{\vect{v}}_{f}^{i+1}-\dot{\vect{v}}_f^{i+\frac 12}) }{\vect{w}_{s}},
	\\
	R_3^{i}(\vect{v}_f)  = &\Thprod{\rho_{12}d_t e_{\vect{v}_{s}}^{I,i+1}+\rho_{22}d_t e_{\vect{v}_{f}}^{I,i+1}}{\vect{w}_f} + \Thprod{\frac{\eta}{\kappa} e_{\vect{v}_{f}}^{I,i+\frac 12}}{\vect{w}_{f}} 
	\\
	& - \Thprod{\rho_{12}(d_t{\vect{v}}_{s}^{i+1} - \dot{\vect{v}}_{s}^{i+\frac 12})+\rho_{22}(d_t{\vect{v}}_{f}^{i+1} - \dot{\vect{v}}_{f}^{i+\frac 12})}{\vect{w}_{f}},
	\\
	R_4^{i}(q)  = &\Thprod{s_{0}d_t e_{p}^{I,i+1}}{q} + \Thprod{\Ael(d_t e_{\sigma}^{I,i+1}+\alpha d_t e_{p}^{I,i+1}\Id)}{\alpha q\Id} - \Thprod{s_{0}(d_t{p}^{i+1}-\dot{p}^{i+\frac 12})}{q} \\
	& - \Thprod{\Ael(d_t{\sigma}^{i+1}- \dot{\sigma}^{i+\frac 12}+\alpha (d_t{p}^{i+1} - \dot{p}^{i+\frac 12})\Id)}{\alpha q\Id} .
\end{align*}
\end{lemma}
\begin{proof}
Recall that the continuous solution satisfies \eqref{eq:variational-eqs} by Lemma~\ref{lemma:consistency}. Consider then the average of the first four equations in \eqref{eq:variational-eqs} with continuous solutions at $t= t_{i}, t_{i+1}$ which are
{
		\begin{align*}
			\Thprod{\Ael(\dot{\sigma}^{i+\frac 12} + \alpha \dot{p}^{i+\frac 12} \Id)}{r} +\Thprod{\divergence r}{\vect{v}_{s}^{i+\frac 12}} -\pThprod{\widehat{\vect{v}}_{s}^{i+\frac 12}}{r \vect{n}} &= 0 ,
			\\ 
			\Thprod{\rho_{11}\dot{\vect{v}}_{s}^{i+\frac 12} + \rho_{12}\dot{\vect{v}}_{f}^{i+\frac 12}}{\vect{w}_{s}} + \Thprod{\sigma^{i+\frac 12}}{\grad \vect{w}_{s}} &
            \\
            \notag 
            - \pThprod{{\sigma^{i+\frac 12} \vect{n}}}{\vect{w}_{s}} & = \Thprod{\vect{f}^{i+\frac 12}}{\vect{w}_{s}}, 
			\\
			\Thprod{\rho_{12}\dot{\vect{v}}_{s}^{i+\frac 12}+\rho_{22}\dot{\vect{v}}_{f}^{i+\frac 12}}{\vect{w}_{f}} + \Thprod{\frac{\eta}{\kappa} \vect{v}_{f}^{i+\frac 12}}{\vect{w}_{f}} - \Thprod{p^{i+\frac 12} }{\divergence\vect{w}_{f}} &
            \\
            \notag 
            +\pThprod{\widehat{p}^{i+\frac 12}}{\vect{w}_{f}\cdot\vect{n}} & = 0, 
			\\
			\Thprod{s_{0}\dot{p}^{i+\frac 12}}{q} + \Thprod{\Ael(\dot{\sigma}^{i+\frac 12}+\alpha \dot{p}^{i+\frac 12}\Id)}{\alpha q\Id} - \Thprod{\vect{v}_{f}}{\grad q} &\\
			+\; \pThprod{\vect{v}_{f}^{i+\frac 12}\cdot \vect{n} + \tau_f(p^{i+\frac 12} - \widehat{p}^{i+\frac 12})}{q} & = \Thprod{g^{i+\frac 12}}{q}
		\end{align*}
}
for any $(r, \vect{w}_s, \vect{w}_f, q, \widehat{\vect{w}}_s, \widehat{q}) \in \Sigma_{h} \times V_{s,h} \times V_{f,h} \times Q_{h} \times \widehat{V}_{s,h}\times \widehat{Q}_{h}$. Rewriting these equations with the discrete time operator gives
%
	\begin{align*}
		&\Thprod{\Ael(d_t \sigma^{i+1} + \alpha d_t{p}^{i+1} \Id)}{r} +\Thprod{\divergence r}{\vect{v}_{s}^{i+\frac 12}} -\pThprod{\widehat{\vect{v}}_{s}^{i+\frac 12}}{r \vect{n}} 
		\\
		\notag
		& = \Thprod{\Ael(d_t \sigma^{i+1} - \dot{\sigma}^{i+\frac 12} + \alpha (d_t{p}^{i+1} - \dot{p}^{i+\frac 12})\Id)}{r} ,
		\\ 
		&\Thprod{\rho_{11}d_t{\vect{v}}_{s}^{i+1}+\rho_{12}d_t{\vect{v}}_{f}^{i+1}}{\vect{w}_{s}} + \Thprod{\sigma^{i+\frac 12}}{\grad \vect{w}_{s}} - \pThprod{\sigma^{i+\frac 12} \vect{n} }{\vect{w}_{s}}
		\\
		& = \Thprod{\vect{f}^{i+\frac 12}}{\vect{w}_{s}} + \Thprod{\rho_{11}(d_t{\vect{v}}_{s}^{i+1}-\dot{\vect{v}}^{i+\frac 12})}{\vect{w}_{s}} + \Thprod{\rho_{12}(d_t{\vect{v}}_{f}^{i+1}-\dot{\vect{v}}_f^{i+\frac 12}) }{\vect{w}_{s}}, 
		\\
		&\Thprod{\rho_{12}d_t{\vect{v}}_{s}^{i+1}+\rho_{22}d_t{\vect{v}}_{f}^{i+1}}{\vect{w}_{f}} + \Thprod{\frac{\eta}{\kappa} \vect{v}_{f}^{i+\frac 12}}{\vect{w}_{f}} 
		- \Thprod{p^{i+\frac 12} }{\divergence\vect{w}_{f}} +\pThprod{\widehat{p}^{i+\frac 12}}{\vect{w}_{f}\cdot\vect{n}} 
		\\
		&= \Thprod{\rho_{12}(d_t{\vect{v}}_{s}^{i+1} - \dot{\vect{v}}_{s}^{i+\frac 12}+\rho_{22}(d_t{\vect{v}}_{f}^{i+1} - \dot{\vect{v}}_{f}^{i+\frac 12})}{\vect{w}_{f}} , 
		\\
		&\Thprod{s_{0}d_t{p}^{i+1}}{q} + \Thprod{\Ael(d_t{\sigma}^{i+1}+\alpha d_t{p}^{i+1}\Id)}{\alpha q\Id} - \Thprod{\vect{v}_{f}^{i+\frac 12}}{\grad q} 
		\\
		&+ \pThprod{\vect{v}_{f}^{i+\frac 12}\cdot \vect{n} + \tau_{f} (p^{i+\frac 12} - \widehat{p}^{i+\frac 12})}{q} 
		\\
		&= \Thprod{ g^{i+\frac 12}}{q} 
		+ \Thprod{s_{0}(d_t{p}^{i+1}-\dot{p}^{i+\frac 12})}{q} + \Thprod{\Ael(d_t{\sigma}^{i+1}- \dot{\sigma}^{i+\frac 12}+\alpha (d_t{p}^{i+1} - \dot{p}^{i+\frac 12})\Id)}{\alpha q\Id} .
	\end{align*}
%
By subtracting \eqref{eq:fully-discrete-cn-eqs} from the last four equations we obtain 
%
	\begin{align*}
		&\Thprod{\Ael(d_t e_{\sigma}^{i+1} + \alpha d_te_{p}^{i+1} \Id)}{r} +\Thprod{\divergence r}{e_{\vect{v}_{s}}^{i+\frac 12}} -\pThprod{\widehat{e}_{\vect{v}_{s}}^{i+\frac 12}}{r \vect{n}} 
		\\
		&\quad = \Thprod{\Ael(d_t \sigma^{i+1} - \dot{\sigma}^{i+\frac 12} + \alpha (d_t{p}^{i+1} - \dot{p}^{i+\frac 12})\Id)}{r} ,
		\\ 
		&\Thprod{\rho_{11}d_t e_{\vect{v}_{s}}^{i+1}+\rho_{12}d_t e_{\vect{v}_{f}}^{i+1}}{\vect{w}_{s}} 
		+ \Thprod{e_{\sigma}^{i+\frac 12} }{\grad \vect{w}_{s}} - \pThprod{e_{\sigma}^{i+\frac 12} \vect{n} - \tau_s ({P_{\widehat{V}_s}} \vect{v}_{s,h}^{i+\frac 12} - \widehat{\vect{v}}_{s,h}^{i+\frac 12})}{\vect{w}_{s}}
		\\
		&\quad=  \Thprod{\rho_{11}(d_t{\vect{v}}_{s}^{i+1}-\dot{\vect{v}}_s^{i+\frac 12})}{\vect{w}_{s}} + \Thprod{\rho_{12}(d_t{\vect{v}}_{f}^{i+1}-\dot{\vect{v}}_f^{i+\frac 12}) }{\vect{w}_{s}}, 
		\\
		&\Thprod{\rho_{12}d_t e_{\vect{v}_{s}}^{i+1}+\rho_{22}d_t e_{\vect{v}_{f}}^{i+1}}{\vect{w}_f} + \Thprod{\frac{\eta}{\kappa} e_{\vect{v}_{f}}^{i+\frac 12}}{\vect{w}_{f}} 
		- \Thprod{e_p^{i+\frac 12} }{\divergence\vect{w}_{f}} +\pThprod{\widehat{e}_p^{i+\frac 12}}{\vect{w}_{f}\cdot\vect{n}} 
		\\
		&\quad = \Thprod{\rho_{12}(d_t{\vect{v}}_{s}^{i+1} - \dot{\vect{v}}_{s}^{i+\frac 12}}{\vect{w}_f} + \Thprod{\rho_{22}(d_t{\vect{v}}_{f}^{i+1} - \dot{\vect{v}}_{f}^{i+\frac 12})}{\vect{w}_{f}} , 
		\\
		&\Thprod{s_{0}d_t e_{p}^{i+1}}{q} + \Thprod{\Ael(d_t e_{\sigma}^{i+1}+\alpha d_t e_{p}^{i+1}\Id)}{\alpha q\Id} - \Thprod{e_{\vect{v}_{f}}^{i+\frac 12}}{\grad q}
		\\
		&\qquad + \pThprod{e_{\vect{v}_{f}}^{i+\frac 12}\cdot \vect{n} + \tau_{f} (e_p^{i+\frac 12} - \widehat{e}_p^{i+\frac 12})}{q}  
		\\
		&\quad = \Thprod{s_{0}(d_t{p}^{i+1}-\dot{p}^{i+\frac 12})}{q} + \Thprod{\Ael(d_t{\sigma}^{i+1}- \dot{\sigma}^{i+\frac 12}+\alpha (d_t{p}^{i+1} - \dot{p}^{i+\frac 12})\Id)}{\alpha q\Id} . 
		\\
	\end{align*}
%
Similarly to the proof of Lemma \ref{eq:interpolationerror_eqns},  we split the error into projection error and projected error in the system above and move the projection error term to the right-hand side. Then the proof follows by the properties of the projection operators.
\end{proof}

Next, we introduce an auxiliary result to the present the main result, the fully discrete error estimates.
\begin{lemma}
\label{lemma:not-gronwall}
Let $\{A_i\}_i, \{B_i\}_i$, $\{E_i\}_i$, and $\{D_i\}_i$ be sequences of nonnegative numbers. Suppose these sequences satisfy
\begin{equation*}
	A_n^2 + \sum_{i=0}^n B_i^2 \leq  A_0^2 + \sum_{i=1}^{n} E_i A_i + \sum_{i=0}^n D_i
\end{equation*}	
for all $n \geq 0$. Then for any $n\geq 0$, there exists $C>0$ independent of $n$ such that
%
	\begin{align*}
		A_n
		&\leq  A_0 +  \sum_{i=1}^{n} E_i + \left({\sum_{i=0}^n D_i}\right)^{1/2} ,
		\\
		\left({\sum_{i=0}^n B_i^2}\right)^{1/2}
		&\leq \;C \left({A_0 +  \sum_{i=1}^{n} E_i + \left({\sum_{i=0}^n D_i }\right)^{1/2}}\right).
	\end{align*}	    
%
\end{lemma}
\begin{proof}
We refer to \cite[Lemma~2]{Cesmelioglu:2023} for its proof.
\end{proof}
\begin{theorem}\label{thm:discrete-error-estimates}
Suppose that $\{(\sigma_{h}^i, \vect{v}_{s,h}^i, \vect{v}_{f,h}^i, p_{h}^i, \widehat{\vect{v}}_{s,h}^i, \widehat{p}_{h}^i)\}_{i=0}^{n}$ is solution of \eqref{eq:fully-discrete-cn-eqs} with numerical initial data satisfying the compatibility conditions \eqref{eq:initial-data-constraint1}, \eqref{eq:initial-data-constraint2}, and the approximations properties \eqref{eq:initial-data-approximation1}, \eqref{eq:initial-data-approximation2}. If exact solutions satisfy the regularity assumptions to make the quantities in the estimate below well-defined, then we have 
\begin{align*}
	\| e_{\sigma}^n + \alpha e_{p}^n\Id \|_{\Ael} + &\|e_{p}^n\|_{s_0} + \rho_0 (\|e_{\vect{v}_{s}}^n\|_{L^2} + \|e_{\vect{v}_{f}}^n\|_{L^2} )
	\;
	\\
	&\le \; C h^m (| \sigma(0), \vect{v}_{f}(0), p(0)|_{m} + |\vect{v}_{s}(0)|_{m})
	\\ 
	&\quad + C h^m \left[  \| \partial_t \sigma, \partial_t p, \partial_t \vect{v}_s, \partial_t \vect{v}_f \|_{L^1(0, t_{n}; H^m)} + | \sigma^n, p^n, \vect{v}_s^n, \vect{v}_f^n |_{m} \right]
	\\
	&\quad + C h^m \left[ \Delta t \sum_{i=0}^{n-1}| \sigma^{i+\frac 12}, p^{i+\frac 12}, \vect{v}_s^{i+\frac 12}, \vect{v}_f^{i+\frac 12} |_{m}^2 \right]^{\frac 12}
	\\
	&\quad + C(\Delta t)^2 \| \partial_t^3 \sigma, \partial_t^3 p, \partial_t^3 \vect{v}_s, \partial_t^3 \vect{v}_f \|_{L^1(0, t_{n}; L^2)} 
\end{align*}
for $1 \le m \le k+1$. Moreover, if we assume that the coefficients satisfy the assumption $\Ael \alpha \Id : \alpha \Id \le C_s s_0$, with a uniform constant $C_s>0$, then $\| e_{\sigma}^n \|_{\Ael}$ satisfies an error estimate with the same right-hand side of the estimate above and a constant depending on $C_s$. 
\end{theorem}
\begin{proof}
By taking $r = 2 e_{\sigma}^{h,i+\frac 12}$, $\vect{w}_{s} = 2e_{\vect{v}_{s}}^{h,i+\frac 12}$, $\vect{w}_{f} = 2e_{\vect{v}_{f}}^{h,i+\frac 12}$, $q = 2e_p^{h,i+\frac 12}$, $\widehat{\vect{w}}_{s} = -2\widehat{e}_{\vect{v}_{s}}^{h,i+\frac 12}$, $\widehat{q}=-2\widehat{e}_{p}^{h,i+\frac 12}$ in \eqref{eq:fully-discrete-reduced-error-eqs} and adding all the equations, 
\begin{align}
	\label{eq:fully-discrete-energy-identity}
	&2 \Thprod{\Ael(d_t e_{\sigma}^{h,i+1} + \alpha d_te_{p}^{h,i+1} \Id)}{e_{\sigma}^{h,i+\frac 12} + \alpha e_p^{h,i+\frac 12}} + 2 \Thprod{s_{0}d_t e_{p}^{h,i+1}}{e_p^{h,i+\frac 12}}
	\\
	\notag
	&\quad + 2\Thprod{\rho_{11}d_t e_{\vect{v}_{s}}^{h,i+1}}{e_{\vect{v}_{s}}^{h,i+\frac 12}} + 2\Thprod{\rho_{12}d_t e_{\vect{v}_{f}}^{h,i+1}}{e_{\vect{v}_{s}}^{h,i+\frac 12}} 
	\\
	\notag
	&\quad + 2\Thprod{\rho_{12}d_t e_{\vect{v}_{s}}^{h,i+1}}{e_{\vect{v}_{f}}^{h,i+\frac 12}} + 2\Thprod{\rho_{22}d_t e_{\vect{v}_{f}}^{h,i+1}}{e_{\vect{v}_{f}}^{h,i+\frac 12}} 
	\\
	\notag
	&\quad + 2 \Thprod{\frac{\eta}{\kappa} e_{\vect{v}_{f}}^{h,i+\frac 12}}{e_{\vect{v}_{f}}^{h,i+\frac 12}} + 2\pThprod{\tau_{s} ({P_{\widehat{V}_s}} e_{\vect{v}_{s}}^{h, i+\frac 12} - \widehat{e}_{\vect{v}_{s}}^{h,i+\frac 12})}{{P_{\widehat{V}_s}} e_{\vect{v}_{s}}^{h,i+\frac 12} - \widehat{e}_{\vect{v}_{s}}^{h,i+\frac 12}} 
	\\
	\notag 
	&\quad + 2\pThprod{\tau_{f} (e_{p}^{h,i+\frac 12} - \widehat{e}_{p}^{h,i+\frac 12})}{e_{p}^{h,i+\frac 12} - \widehat{e}_{p}^{h,i+\frac 12}}
	\\
	\notag
	&= 2 R_1^i(e_{\sigma}^{h,i+\frac 12}) + 2 R_2^i(e_{\vect{v}_s}^{h,i+\frac 12}) + 2 R_3^i(e_{\vect{v}_f}^{h,i+\frac 12}) + 2 R_4^i(e_{p}^{h,i+\frac 12}) 
	\\
	\notag 
	&\quad + \Thprod{e_{\sigma}^{I,i+\frac 12}}{2 \grad e_{\vect{v}_{s}}^{h,i+\frac 12}} - \pThprod{e_{\sigma}^{I,i+\frac 12} \vect{n} - \tau_s (\Pi_{\vect{v}_s} \vect{v}_{s}^{i+\frac 12} - {P_{\widehat{V}_s}}\widehat{\vect{v}}_{s}^{i+\frac 12})}{2e_{\vect{v}_{s}}^{h,i+\frac 12}} 
	\\
	\notag 
	&\quad - \pThprod{e_{\sigma}^{I,i+\frac 12} \vect{n} - \tau_{s} ({P_{\widehat{V}_s}} e_{\vect{v}_{s}}^{I,i+\frac 12} - \widehat{e}_{\vect{v}_{s}}^{I,i+\frac 12})}{\widehat{e}_{\vect{v}_{s}}^{h, i+\frac 12}}
	\\
	\notag
	&= 2 R_1^i(e_{\sigma}^{h,i+\frac 12}) + 2 R_2^i(e_{\vect{v}_s}^{h,i+\frac 12}) + 2 R_3^i(e_{\vect{v}_f}^{h,i+\frac 12}) + 2 R_4^i(e_{p}^{h,i+\frac 12}) 
	\\
	\notag
	&\quad + 2 \pThprod{R(\sigma^{i+\frac 12}, \vect{v}_s^{i+\frac 12})}{{P_{\widehat{V}_s}} e_{\vect{v}_s}^{h,i+\frac 12} - \widehat{e}_{\vect{v}_s}^{i+\frac 12}}
\end{align}
where the last formula follows by the properties of the projection operator in Theorem~\ref{thm:elasticity-interpolation}. 
Let us define $X_i$, $Z_i$, $Y_i$ by 
\begin{align*}
	X_i^2 &= \Thprod{\Ael(e_{\sigma}^{h,i} + \alpha e_{p}^{h,i} \Id)}{e_{\sigma}^{h,i} + \alpha e_{p}^{h,i} \Id)} +\Thprod{s_{0}e_{p}^{h,i}}{e_{p}^{h,i}}
	\\
	&\quad + \Thprod{\rho_{11} e_{\vect{v}_{s}}^{h,i}}{e_{\vect{v}_{s}}^{h,i}} + \Thprod{\rho_{12} e_{\vect{v}_{f}}^{h,i}}{e_{\vect{v}_{s}}^{h,i}} + \Thprod{\rho_{12} e_{\vect{v}_{s}}^{h,i}}{e_{\vect{v}_{f}}^{h,i}} + \Thprod{\rho_{22} e_{\vect{v}_{f}}^{h,i}}{e_{\vect{v}_{f}}^{h,i}} ,
	\\
	Z_i^2 &=  \Thprod{\frac{\eta}{\kappa} e_{\vect{v}_{f}}^{h,i+\frac 12}}{e_{\vect{v}_{f}}^{h,i+\frac 12}} + \pThprod{\tau_{s} ({P_{\widehat{V}_s}} e_{\vect{v}_{s}}^{h, i+\frac 12} - \widehat{e}_{\vect{v}_{s}}^{h,i+\frac 12})}{{P_{\widehat{V}_s}} e_{\vect{v}_{s}}^{h,i+\frac 12} - \widehat{e}_{\vect{v}_{s}}^{h,i+\frac 12}} 
	\\
	&\quad + \pThprod{\tau_{f} (e_{p}^{h,i+\frac 12} - \widehat{e}_{p}^{h,i+\frac 12})}{e_{p}^{h,i+\frac 12} - \widehat{e}_{p}^{h,i+\frac 12}},
	\\
	Y_i^2 &= \Delta t \sum_{j=0}^{i-1}Z_j^2 .
\end{align*}
Then, \eqref{eq:fully-discrete-energy-identity} is 
\begin{align*}
	\frac{1}{\Delta t} (X_{i+1}^2 - X_i^2) + 2 Z_i^2 &= 2 R_1^i(e_{\sigma}^{h,i+\frac 12}) + 2 R_2^i(e_{\vect{v}_s}^{h,i+\frac 12}) + 2 R_3^i(e_{\vect{v}_f}^{h,i+\frac 12}) + 2 R_4^i(e_{p}^{h,i+\frac 12}) 
	\\
	&\quad + 2\pThprod{R(\sigma^{i+\frac 12}, \vect{v}_s^{i+\frac 12})}{{P_{\widehat{V}_s}} e_{\vect{v}_{s}}^{h,i+\frac 12} - \widehat{e}_{\vect{v}_{s}}^{h,i+\frac 12}}.
\end{align*}
The summation of this identity from $i=0$ to $i=n-1$ gives
\begin{align}
	\notag
	X_n^2 - X_0^2 + 2 Y_n^2 &= 2 \Delta t \sum_{i=0}^{n-1} (R_1^i(e_{\sigma}^{h,i+\frac 12}) + R_2^i(e_{\vect{v}_s}^{h,i+\frac 12}) + R_3^i(e_{\vect{v}_f}^{h,i+\frac 12}) + R_4^i(e_{p}^{h,i+\frac 12}))
	\\
	\label{eq:fully-discrete-sum1}
	&\quad + 2\Delta t \sum_{i=0}^{n-1} \pThprod{R(\sigma^{i+\frac 12}, \vect{v}_s^{i+\frac 12})}{{P_{\widehat{V}_s}} e_{\vect{v}_{s}}^{h,i+\frac 12} - \widehat{e}_{\vect{v}_{s}}^{h,i+\frac 12}}.
\end{align}
By the inequality 
	\begin{align*}
		\|e_{\xi}^{I,i+1} - e_{\xi}^{I,i}\|_{L^2(\Omega)} &\le  \int_{t_i}^{t_{i+1}} \left\|\partial_t \xi (s) - \Pi_{\xi} \partial_t \xi(s)\right\|_{L^2(\Omega)}\,ds & & \text{ for } \xi = \sigma, \vect{v}_s, \vect{v}_f, p, 
	\end{align*}
	with \eqref{eq:elasticity-interpolation-estimate}, \eqref{eq:darcy-interpolation-estimate}, and by the time discretization error approximation (cf.~\cite[Lemma~6]{Lee-Bolanos:2025})
	\begin{align*}
		\|\xi^{i+1} - \xi^i - \Delta t \dot{\xi}^{i+\frac 12}\|_{L^2(\Omega)} &\le \frac{(\Delta t)^2}{2} \|\partial_t^3 \xi \|_{L^1(t_i, t_{i+1}; L^2(\Omega))} \quad \text{ for } \xi = \sigma, \vect{v}_s, \vect{v}_f, p,
	\end{align*}
we have 
\begin{align}
	\label{eq:Ri-estimate1}
	&2\Delta t |R_1^i(e_{\sigma}^{h,i+\frac 12}) + R_4^i(e_{p}^{h,i+\frac 12}))|
	\\
	\notag
	&\quad \le C( h^m \| {\partial_t \sigma, \partial_t \vect{v}_s, \partial_t \vect{v}_f, \partial_t p} \|_{L^1(t_i, t_{i+1}; H^m)} + (\Delta t)^2 \| \partial_t^3 \sigma, \partial_t^3 p \|_{L^1(t_i, t_{i+1}; L^2)}) 
    \\
    \notag
    &\qquad \times \| e_{\sigma}^{h,i+\frac 12} + \alpha e_{p}^{h,i+\frac 12}\Id \|_{\Ael} 
	\\
	\notag
	&\qquad + C( h^m \| {\partial_t \vect{v}_f}, \partial_t p \|_{L^1(t_i, t_{i+1}; H^m)} + (\Delta t)^2 \| \partial_t^3 p \|_{L^1(t_i, t_{i+1}; L^2)})\|e_{p}^{h,i+\frac 12}\|_{s_0}
	\\
	\notag
	&\quad \le  C( h^m \| {\partial_t \sigma, \partial_t \vect{v}_s, \partial_t \vect{v}_f, \partial_t p} \|_{L^1(t_i, t_{i+1}; H^m)} + (\Delta t)^2 \| \partial_t^3 \sigma, \partial_t^3 p \|_{L^1(t_i, t_{i+1}; L^2)}) (X_i + X_{i+1}) 
	\\
	\notag 
	& \quad =: E_{a,i}(X_{i} + X_{i+1}),
	\\
	\label{eq:Ri-estimate2}
	&2\Delta t |R_2^i(e_{\vect{v}_s}^{h,i+\frac 12}) + R_3^i(e_{\vect{v}_f}^{h,i+\frac 12}))|
	\\
	\notag
	&\quad \le C( h^m \| {\partial_t \sigma, \partial_t \vect{v}_s, \partial_t \vect{v}_f, \partial_t p} \|_{L^1(t_i, t_{i+1}; H^m)} + (\Delta t)^2 \| \partial_t^3 \vect{v}_s, \partial_t^3 \vect{v}_f \|_{L^1(t_i, t_{i+1}; L^2)}) 
    \\
    \notag
    &\qquad \times \rho_0 (\| e_{\vect{v}_s}^{h,i+\frac 12}\|_{L^2} + \| e_{\vect{v}_f}^{h,i+\frac 12} \|_{L^2}) + C \Delta t h^m \| \vect{v}_f^i, \vect{v}_f^{i+1} \|_{H^m} \| e_{\vect{v}_f}^{h,i+\frac 12} \|_{L^2} 
	\\
	\notag
	&\quad \le C( h^m \| {\partial_t \sigma, \partial_t \vect{v}_s, \partial_t \vect{v}_f, \partial_t p} \|_{L^1(t_i, t_{i+1}; H^m)} + (\Delta t)^2 \| \partial_t^3 \vect{v}_s, \partial_t^3 \vect{v}_f \|_{L^1(t_i, t_{i+1}; L^2)}) (X_i + X_{i+1}) 
	\\
	\notag
	&\qquad + C \Delta t h^m | \vect{v}_f^{i+\frac 12} |_{H^m} Z_i ,
	\\
	& \notag 
	\quad =: E_{b,i} (X_i + X_{i+1}) + C \Delta t h^m | \vect{v}_f^{i+\frac 12} |_{H^m} Z_i ,
	\\
	\label{eq:Ri-estimate3}
	&\Delta t |\pThprod{R(\sigma^{i+\frac 12}, \vect{v}_s^{i+\frac 12})}{{P_{\widehat{V}_s}} e_{\vect{v}_{s}}^{h,i+\frac 12} - \widehat{e}_{\vect{v}_{s}}^{h,i+\frac 12}}| 
	\\
	\notag
	&\quad \le C\Delta t h^m (|\sigma^{i+\frac 12}|_{m} + |\vect{v}_s^{i+\frac 12}|_{m}) \|\tau_s^{\frac 12} (P_{\widehat{V}_s} e_{\vect{v}_{s}}^{h,i+\frac 12} - \widehat{e}_{\vect{v}_{s}}^{h,i+\frac 12})\|_{L^2(\mathcal{F}_h)}
	\\
	\notag
	&\quad \le C\Delta t h^m (|\sigma^{i+\frac 12}|_{m} + |\vect{v}_s^{i+\frac 12}|_{m}) Z_i .
\end{align}
By \eqref{eq:fully-discrete-sum1}, \eqref{eq:Ri-estimate1}, \eqref{eq:Ri-estimate2}, \eqref{eq:Ri-estimate3} with Young's inequality, 
\begin{align*}
	X_n^2 + Y_n^2 \le X_0^2 + \sum_{i=1}^{n-1} (E_{a,i}+E_{b,i}) X_i + C \sum_{i=0}^{n-1} \Delta t h^{2m} (|\sigma^{i+\frac 12}|_m^2 + |\vect{v}_s^{i+\frac 12}|_m^2 + |\vect{v}_f^{i+\frac 12}|_m^2), 
\end{align*}
and by Lemma~\ref{lemma:not-gronwall}, 
\begin{align*}
	X_n \le X_0 + \sum_{i=1}^{n-1} (E_{a,i}+E_{b,i}) + Ch^m \left( \sum_{i=0}^{n-1} \Delta t (|\sigma^{i+\frac 12}|_m^2 + |\vect{v}_s^{i+\frac 12}|_m^2 + |\vect{v}_f^{i+\frac 12}|_m^2) \right)^{\frac 12} .
\end{align*}
The conclusion follows by the triangle inequality, \eqref{eq:darcy-interpolation-estimate}, and \eqref{eq:elasticity-interpolation-estimate}.
\end{proof}


\section{Numerical examples}
\label{sec:num_test}
{
In this section, we present several numerical examples that demonstrate the properties of the fully discrete scheme \eqref{eq:fully-discrete-cn-eqs}. First, in Example 1, we perform convergence tests using smooth manufactured solutions to confirm the optimal \textit{a priori} error estimates for both the semidiscrete and fully discrete schemes. Compressible and nearly incompressible isotropic materials are considered to illustrate the locking-free property of the method. In the next two examples we simulate physically realistic scenarios of poroelastic wave propagation, confirming the method's applicability to complex geophysical models. In Example 2, we present simulation in both isotropic and anisotropic media, and finally in Example we simulate in heterogeneous media. All experiments in this section were implemented using the open source finite element software NETGEN/NGSolve \cite{schoberl1997netgen,schoberl2014c++}.
}

\subsubsection*{Computational Efficiency: Degrees of Freedom Comparison}

To explicitly address the computational advantages of the proposed HDG formulation over standard Discontinuous Galerkin (DG) methods, we perform a comparison of the globally coupled Degrees of Freedom (DOFs). A well-known drawback of standard DG methods is the rapid growth of the global system size when the polynomial degree $k$ is increased, as all volumetric variables remain globally coupled.

In our HDG approach, all volumetric approximations (local variables $\sigma_h, \vect{v}_{s,h}, \vect{v}_{f,h}, p_h$) are strictly local to each element. Through static condensation, these local DOFs are eliminated from the global system. Consequently, the only globally coupled unknowns are the hybridization trace variables ($\widehat{\vect{v}}_{s,h}, \widehat{p}_h$). Notably, the hybridization variable for solid velocity is approximated by shape functions of degree $k$, one degree lower than its volumetric counterpart.

Table \ref{tab:dof_comparison} presents a theoretical comparison of the global system size for a 2D simplicial mesh with $N_e$ elements and $N_f$ faces. The table highlights the substantial reduction in the global matrix size achieved by the HDG method compared to DG approximation of the local variables $\sigma_h, \vect{v}_{s,h}, \vect{v}_{f,h}, p_h$) using discontinuous piecewise polynomial spaces of degree $k$ as presented in \cite{laPuente-Dumbser-Kaser-Igel:2008}. For high-order approximations ($k=3, 4$), the HDG formulation reduces the global system size by 77\% or more compared to a standard DG approach, demonstrating its exceptional computational efficiency for wave propagation problems where high resolution is required.

\begin{table}[htpb]
    \centering
    \renewcommand{\arraystretch}{1.2}
    \caption{Comparison of globally coupled Degrees of Freedom (DOFs) between a standard DG approach and the proposed HDG method in 2D. The standard DG system considers local approximations $\sigma_h, \vect{v}_{s,h}, \vect{v}_{f,h}, p_h$  of degree $k$. The HDG global system depends only on the trace variables of degree $k$ defined on the mesh skeleton ($N_f$ faces). The reduction percentage is estimated assuming a typical 2D simplicial mesh where $N_f \approx 1.5 N_e$.}
    \begin{tabular}{c c c c}
        \toprule
        $k$ & Total Vol. DOFs & HDG DOFs & Reduction \\
        \midrule
        1 & $24 N_e$  & $6 N_f$  & $\sim 62.5\%$ \\
        2 & $48 N_e$  & $9 N_f$  & $\sim 71.9\%$ \\
        3 & $80 N_e$  & $12 N_f$ & $\sim 77.5\%$ \\
        4 & $120 N_e$ & $15 N_f$ & $\sim 81.3\%$ \\
        \bottomrule
    \end{tabular}
    \label{tab:dof_comparison}
\end{table}

\subsection*{Example 1: History of convergence test}\label{convergenceexamples}
Here we validate the predictions given in Theorem \ref{thm:semidiscrete-error-estimate} and Theorem \ref{thm:discrete-error-estimates} for the HDG method with Crank-Nicolson scheme presented in \eqref{eq:fully-discrete-cn-eqs}. We use the index $k=1,2$ for the polynomial order in the HDG approximation spaces \eqref{eq:HDGspaces}. For each of the approximations $\sigma_{h}$, $\vect{v}_{s,h}$, $\vect{v}_{f,h}$, and $p_{h}$, we compute the $L^{2}$ errors of the approximations for each volumetric variable in the last time step, and then estimate their orders of convergence (e.o.c.). For example, for the stress we have
\[
\mbox{error}(h)=  \| \sigma(t^{n}) - \sigma_{h}^{n}\|_{L^{2}(\Omega)^{2\times 2}},\qquad \mbox{e.o.c}(h) = \frac{ \log( \mbox{error}(h) / \mbox{error}(h')  )}{ \log(h / h')  },
\]
where $h'$ corresponds to the previous mesh size parameter in the computations. 

In this test, we use uniform triangulations with the mesh size parameter $h$ in the computational domain $\Omega=(0,1)^2$ and consider the manufactured solution of the problem \eqref{eq:modelproblem}  given by
\begin{align*}
\vect{u}_{s}(x,y,t) \,\,&=\,\, 
\begin{pmatrix}
    \sin(\pi x) \sin(\pi y)\sin(\pi t) \\
    x y (x - 1) (y - 1)\sin(\pi t)
\end{pmatrix},
\qquad 
  p = x(1 - x)\sin^{2}(\pi y)(2+\cos(\pi t)).
\end{align*}
The material parameters are set to: $ (\rho_{11},\rho_{12}, \rho_{22}) =(1, 1, 2)$,  $(\eta, \kappa, \alpha) = (1,1,1)$, $s_0=1$.  For this example, we test with two sets of elastic constants; first with Young modulus and Poisson ratio $(E, \nu)=(3,0.3)$, and then with a nearly incompressible material $(E, \nu) = (3, 0.499)$. The remaining data for the problem is obtained by substituting the manufactured solution into the equations.

Table \ref{tab:Ex1} presents the errors and estimated convergence orders for polynomial degrees $k=1$ and $k=2$, and for the material parameters $(E, \nu) = (3, 0.3)$.  The results demonstrate an optimal order of convergence of $k+1$ for the approximations of stress ($\sigma$), fluid velocity ($\vect{v}_{f}$), and pressure ($p$). Nonetheless the results in Theorem \ref{thm:semidiscrete-error-estimate} and Theorem \ref{thm:discrete-error-estimates} only cover a convergence order $(k+1)$ of $\vect{v}_s$ errors, we observe a convergence order of $(k+2)$ which is optimal for polynomial degree $k+1$ of $V_{s,h}$.  In Table \ref{tab:Ex1_near_incompressible}, we present the convergence history for the nearly incompressible material with $(E, \nu) = (3, 0.499)$. The same orders of convergence are observed, illustrating the locking-free property of the HDG method.


\begin{table}[htbp!]
\small
\centering
\caption{Example 1. History of convergence test. Material parameters $(E,\nu) = (3,0.3)$.}
\begin{tabular}{@{\hskip .1in}ccl@{\hskip .1in}cl@{\hskip .1in}cc@{\hskip .1in}cc@{\hskip .1in}cc} \toprule
& & &\multicolumn{2}{c}{$\sigma_{h}$}&\multicolumn{2}{c}{$\vect{v}_{s,h}$}&\multicolumn{2}{c}{$\vect{v}_{f,h}$}&\multicolumn{2}{c}{$p_{h}$}\\ 
\cmidrule(lr){4-5} \cmidrule(lr){6-7} \cmidrule(lr){8-9} \cmidrule(lr){10-11}
$k$ & $h$ & & error & e.o.c. & error & e.o.c. & error & e.o.c.& error & e.o.c.\\ \midrule
\multirow{7}{*}{1} 
& $2^{-1}$   & & 3.40e-2 & $-$  & 9.44e-2 & $-$  & 5.63e-1 & $-$  & 1.77e+0 & $-$ \\
& $2^{-2}$   & & 7.36e-3 & 2.21 & 3.20e-2 & 1.56 & 2.82e-1 & 1.00 & 3.80e-1 & 2.22 \\
& $2^{-3}$   & & 1.67e-3 & 2.14 & 6.82e-3 & 2.23 & 9.95e-2 & 1.51 & 8.28e-2 & 2.20 \\
& $2^{-4}$   & & 4.09e-4 & 2.03 & 1.39e-3 & 2.30 & 3.05e-2 & 1.71 & 1.67e-2 & 2.31 \\
& $2^{-5}$   & & 1.09e-4 & 1.91 & 1.68e-4 & 3.05 & 8.44e-3 & 1.85 & 3.71e-3 & 2.17 \\
& $2^{-6}$   & & 2.80e-5 & 1.96 & 1.95e-5 & 3.10 & 2.18e-3 & 1.95 & 8.99e-4 & 2.04 \\
& $2^{-7}$   & & 7.07e-6 & 1.98 & 2.44e-6 & 3.00 & 5.42e-4 & 2.01 & 2.24e-4 & 2.01 \\
\midrule
\multirow{7}{*}{2} 
& $2^{-1}$   & & 9.65e-3 & $-$  & 2.65e-2 & $-$  & 3.46e-1 & $-$  & 4.39e-1 & $-$  \\
& $2^{-2}$   & & 1.16e-3 & 3.06 & 2.92e-3 & 3.18 & 7.37e-2 & 2.23 & 4.88e-2 & 3.17 \\
& $2^{-3}$   & & 1.46e-4 & 2.99 & 1.59e-4 & 4.20 & 1.40e-2 & 2.39 & 5.01e-3 & 3.29 \\
& $2^{-4}$   & & 1.80e-5 & 3.02 & 1.43e-5 & 3.48 & 2.36e-3 & 2.57 & 5.70e-4 & 3.13 \\
& $2^{-5}$   & & 2.24e-6 & 3.01 & 8.29e-7 & 4.10 & 3.56e-4 & 2.73 & 6.98e-5 & 3.03 \\
& $2^{-6}$   & & 2.79e-7 & 3.00 & 5.01e-8 & 4.05 & 4.90e-5 & 2.86 & 8.75e-6 & 2.99 \\
& $2^{-7}$   & & 3.48e-8 & 3.00 & 4.20e-9 & 3.58 & 6.31e-6 & 2.96 & 1.10e-6 & 3.00 \\
\midrule
\multirow{5}{*}{3}
& $2^{-1}$   & & 1.48e-3 & $-$  & 8.87e-2 & $-$  & 8.87e-2 & $-$  &	1.07e-1 & $-$ \\
& $2^{-2}$   & & 8.07e-5 & 4.19	& 1.67e-3 &	5.73 & 1.21e-2 & 2.87 &	3.42e-3 & 4.96 \\
& $2^{-3}$   & & 5.10e-6 & 3.99	& 2.66e-5 &	5.97 & 1.16e-3 & 3.38 &	1.61e-4 & 4.41 \\
& $2^{-4}$   & & 3.24e-7 & 3.97 & 8.30e-7 &	5.00 & 9.60e-5 & 3.59 &	1.09e-5 & 3.88 \\
& $2^{-5}$   & & 2.04e-8 & 3.99 & 3.43e-8 &	4.60 & 6.75e-6 & 3.83 &	7.93e-7	& 3.78 \\
\bottomrule
\end{tabular}
\label{tab:Ex1}
\end{table}




\begin{table}[htbp!]
\small
\centering
\caption{Example 1. History of convergence test. Nearly incompressible material $(E,\nu) = (3,0.499)$.}
\begin{tabular}{@{\hskip .1in}ccl@{\hskip .1in}cl@{\hskip .1in}cc@{\hskip .1in}cc@{\hskip .1in}cc} \toprule
& & &\multicolumn{2}{c}{$\sigma_{h}$}&\multicolumn{2}{c}{$\vect{v}_{s,h}$}&\multicolumn{2}{c}{$\vect{v}_{f,h}$}&\multicolumn{2}{c}{$p_{h}$}\\ 
\cmidrule(lr){4-5} \cmidrule(lr){6-7} \cmidrule(lr){8-9} \cmidrule(lr){10-11}
$k$ & $h$ & & error & e.o.c. & error & e.o.c. & error & e.o.c.& error & e.o.c.\\ \midrule
\multirow{7}{*}{1} 
 & $2^{-1}$    & & 1.53e-4 & $-$  & 8.50e+0 & $-$  & 9.77e-1 & $-$  & 1.85e+0 & $-$  \\
 & $2^{-2}$    & & 3.84e-5 & 1.99 & 4.93e+0 & 0.79 & 3.73e-1 & 1.39 & 4.53e-1 & 2.03 \\
 & $2^{-3}$    & & 1.01e-5 & 1.93 & 1.38e+0 & 1.83 & 1.23e-1 & 1.60 & 1.02e-1 & 2.16 \\
 & $2^{-4}$    & & 3.09e-6 & 1.71 & 3.20e-1 & 2.11 & 3.38e-2 & 1.86 & 2.20e-2 & 2.21 \\
 & $2^{-5}$    & & 9.53e-7 & 1.70 & 3.88e-2 & 3.04 & 8.58e-3 & 1.98 & 4.28e-3 & 2.36 \\
 & $2^{-6}$    & & 2.57e-7 & 1.89 & 4.35e-3 & 3.16 & 2.19e-3 & 1.97 & 9.19e-4 & 2.22 \\
 & $2^{-7}$    & & 6.62e-8 & 1.96 & 5.14e-4 & 3.08 & 5.43e-4 & 2.01 & 2.24e-4 & 2.04 \\
\midrule
\multirow{7}{*}{2} 
 & $2^{-1}$   & & 3.47e-5 & $-$  & 2.37e+0 & $-$  & 4.15e-1 & $-$  & 4.38e-1 & $-$  \\
 & $2^{-2}$   & & 5.19e-6 & 2.74 & 4.50e-1 & 2.39 & 7.97e-2 & 2.38 & 4.89e-2 & 3.16 \\
 & $2^{-3}$   & & 6.87e-7 & 2.92 & 6.05e-2 & 2.90 & 1.46e-2 & 2.45 & 5.02e-3 & 3.28 \\
 & $2^{-4}$   & & 8.58e-8 & 3.00 & 1.87e-3 & 5.02 & 2.36e-3 & 2.63 & 5.70e-4 & 3.14 \\
 & $2^{-5}$   & & 1.08e-8 & 2.99 & 3.03e-4 & 2.63 & 3.57e-4 & 2.73 & 6.98e-5 & 3.03 \\
 & $2^{-6}$   & & 1.35e-9 & 3.00 & 8.50e-6 & 5.15 & 4.90e-5 & 2.86 & 8.75e-6 & 3.00 \\
 & $2^{-7}$   & & 1.69e-10 & 3.00 & 5.48e-7 & 3.96 & 6.31e-6 & 2.96 & 1.10e-6 & 2.99 \\
\midrule
\multirow{5}{*}{3} 
 & $2^{-1}$   & &3.31e-6      & $-$  &  1.60e-1   & $-$  & 9.01e-2   & $-$  & 6.72e-2  & $-$ \\
 & $2^{-2}$   & &2.09e-7      & 3.99 &  2.46e-2   & 2.70 & 1.24e-2   & 2.89 & 2.98e-3  & 4.49 \\
 & $2^{-3}$   & &1.72e-8      & 3.60 &  9.62e-4   & 4.68 & 1.16e-3   & 3.39 & 1.69e-4  & 4.15 \\
 & $2^{-4}$   & &1.30e-9      & 3.72 &  2.73e-5   & 5.14 & 9.60e-5   & 3.60 & 1.15e-5  & 3.87 \\
 & $2^{-5}$   & &8.71e-11      & 3.91 &  7.91e-7   & 5.11 & 6.75e-6   & 3.83 & 6.88e-7  & 4.07 \\
 \bottomrule
\end{tabular}
\label{tab:Ex1_near_incompressible}
\end{table}

\subsection*{Example 2.: Propagation on isotropic and anisotropic poroelastic  media}
In this example, we present two simulations of poroelastic wave propagation on isotropic and anisotropic media. We replicate the test presented in \cite{Alkhimenkov:2021}. Similar tests are also presented in \cite{Lemoine-Ou-LeVeque:2013} and \cite{laPuente-Dumbser-Kaser-Igel:2008}. 

The computational domain is the two-dimensional rectangle $(-4.675, 4.675)^{2}$ and the material parameters are specified in Table \ref{tab:materials}, specifically we compute using the parameters corresponding to the isotropic sandstone and the anisotropic epoxy glass materials. In the isotropic case, a Gaussian pulse is imposed as the initial condition of the second component of the solid velocity. For the anisotropic case, the Gaussian pulse is imposed in the component $yy$ of the stress and also in the pressure. The Gaussian function used in both cases is given by $A_0 \exp\left(-\left( (x/l_x)^{2} + (y/l_y)^2\right) \right)$, with $A_0 = 1$, and $l_x=l_y = 0.08.$ For both cases, we use triangulation with mesh size parameter $h=9.35/200$  and refine at the origin to properly resolve the initial condition.

\begin{table}[htbp!] \small 
\centering
\caption{Parameters used in the simulations of isotropic, anisotropic and heterogeneous poroelastic models.}
\begin{tabular}{@{}ll@{\hskip 0.2in}cc@{}cc@{}cc@{}}
\toprule
\multicolumn{2}{l}{Material} & Isotropic &&  Anisotropic && \multicolumn{2}{c}{Heterogeneous} 
\\ \cmidrule{3-3} \cmidrule{5-5} \cmidrule{7-8} 
Parameters & Units & Sandstone && Glass-epoxy && Sandstone & Shale \\
\midrule
$c_{11}$    & (GPa)                 & 36     && 39.4   && 36     & 11.9 \\
$c_{13}$    & (GPa)                 & 12     && 1.2    && 12     & 3.9 \\
$c_{33}$    & (GPa)                 & 36     && 13.1   && 36     & 11.9 \\
$c_{55}$    & (GPa)                 & 12     && 3.0    && 12     & 3.9 \\
$s_0$       &  $({\rm GPa})^{-1}$   & 8.75e-2&& 9.8e-2 && 8.75e-2& 6.03e-2 \\
$\alpha$    &  -                    & 0.5    && 0.92   && 0.5    & 0.13\\
$\rho_{11}$ & ($\rm kg/m^3$)        & 2208   && 1660   && 2208   & 2022.8 \\
$\rho_{12}$ & ($\rm kg/m^3$)        & 1040   && 1040   && 1040   & 1040 \\
$(\rho_{22})_{11}$ & ($\rm kg/m^3$) & 10400  && 10400  && 10400  & 13000\\
$(\rho_{22})_{33}$ & ($\rm kg/m^3$) & 18720  && 18720  && 10400  & 13000 \\
$\kappa_{11}$      & ($\rm m^2$)    & 6e{-13}&& 6e{-13}&& 6e{-13} & 1e{-13}  \\
$\kappa_{33}$      & ($\rm m^2$)    & 1e{-13}&& 1e{-13}&& 6e{-13} & 1e{-13}\\
$\eta$      & ($\rm kg/m\cdot s$)   & 1e{-3} && 1e{-3} && 0& 0\\
$\phi$      &  -                    & 0.2    && 0.2    && 0.2    & 0.16\\
$K_s$       & (GPa)                 & 40     && 40     && 40     &  7.6\\
$K_f$       & (GPa)                 & 2.5    && 2.5    && 2.5    &  2.5\\
$K$         & (GPa)                 & 20     && 3.2    && 20     & 6.6\\
\bottomrule
\end{tabular}
\label{tab:materials}
\end{table}


\begin{figure}[htb]
    \centering
    \includegraphics[width=0.49\linewidth]{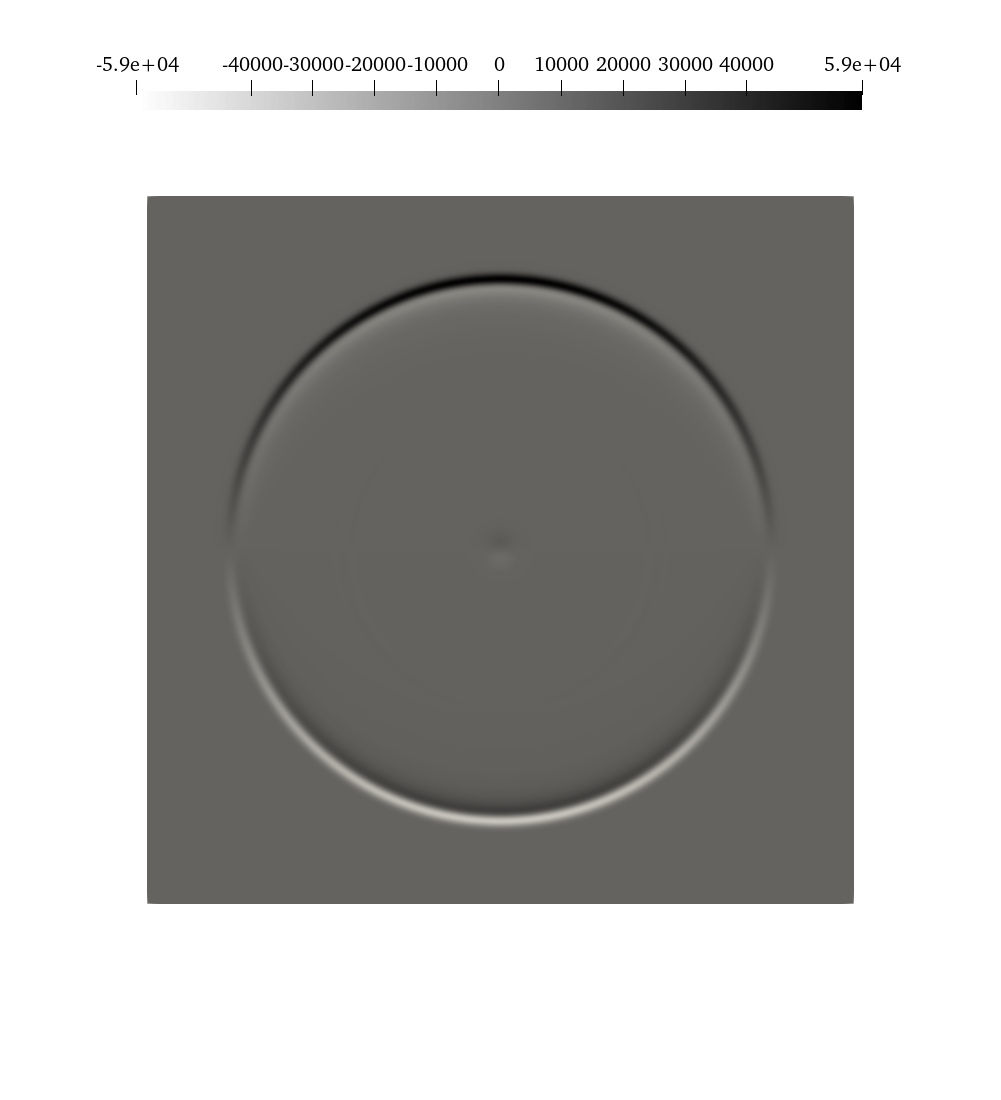}
    \includegraphics[width=0.49\linewidth]{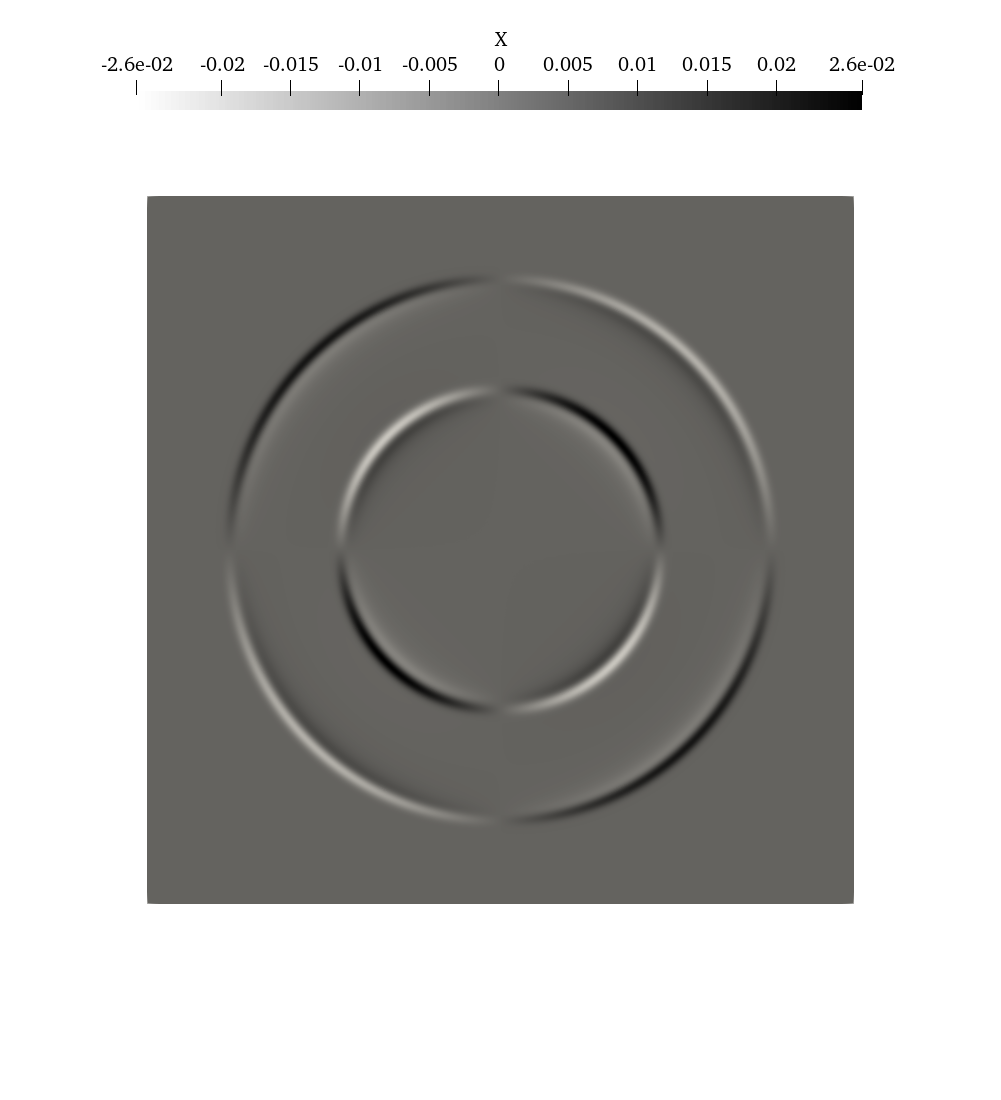}
    \caption{Example 2. Simulation of a Gaussian pulse on isotropic sandstone media. Left: Fluid pressure approximation $p_h$. Right: First component of the solid velocity approximation $(\vect{v}_{s,h})_1$. }
    \label{fig:ex2isotropic}
\end{figure}        
%
\begin{figure}[htb]
    \centering
    \includegraphics[width=0.49\linewidth]{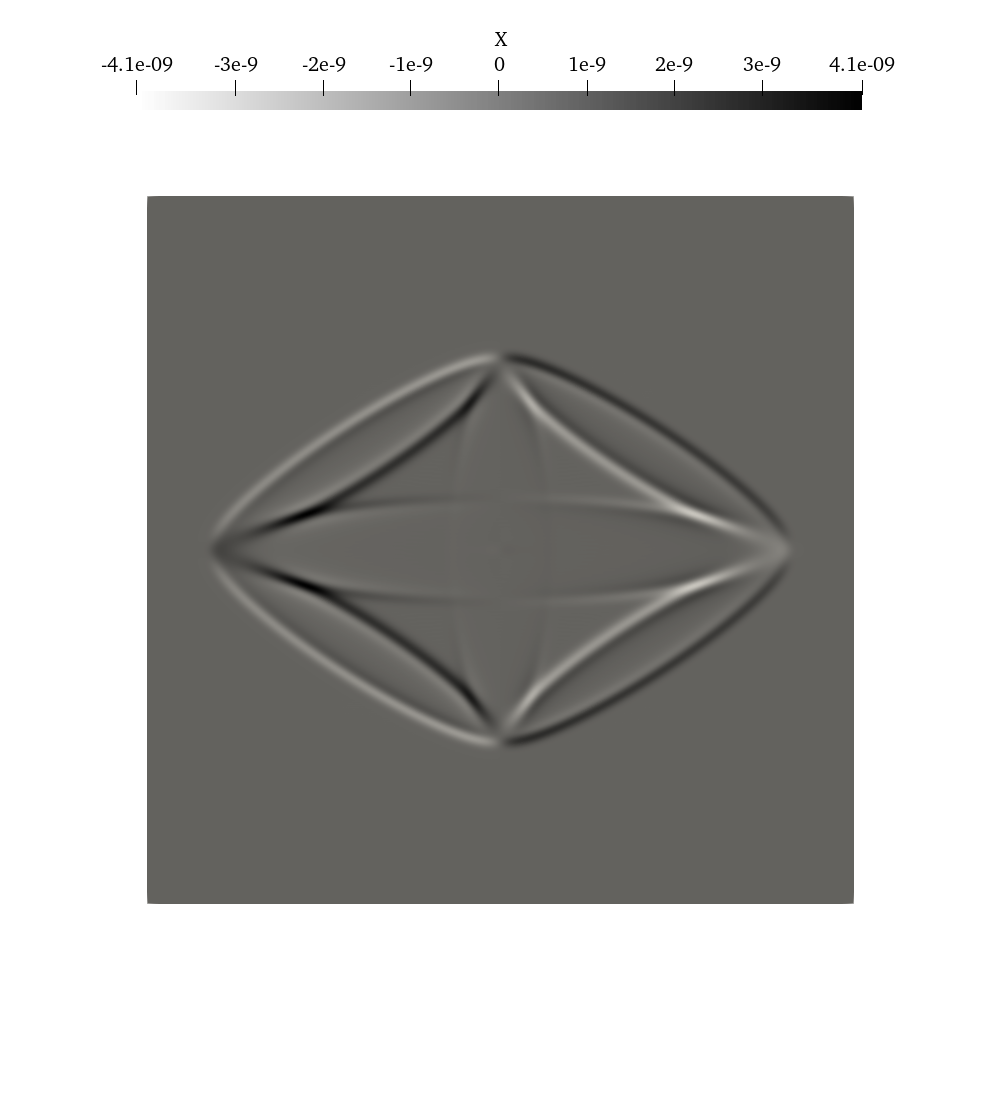}
    \includegraphics[width=0.49\linewidth]{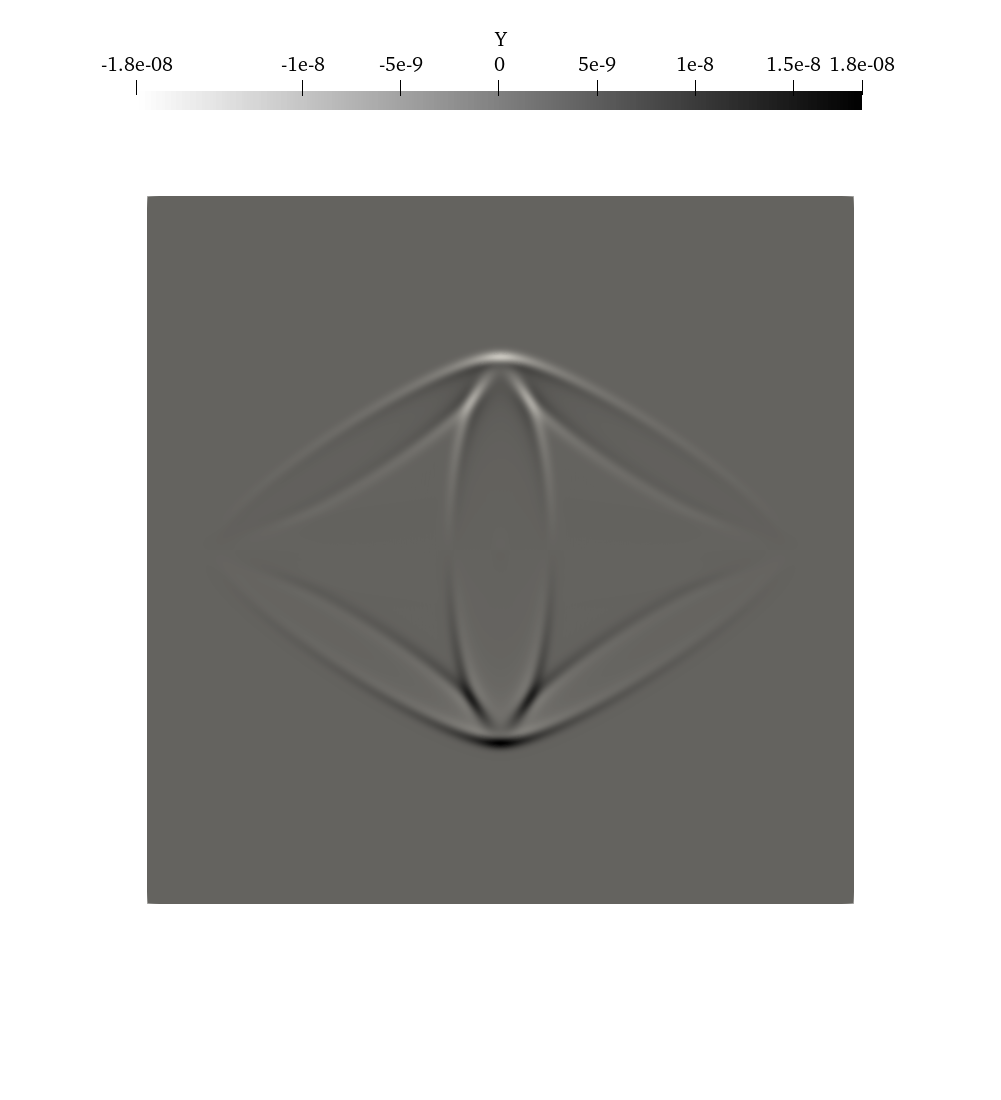}
    \caption{{Example 2. Simulation of a Gaussian pulse on anisotropic epoxy glass media. Left: First component of the solid velocity approximation $(\vect{v}_{s,h})_1$. Right: Second component of the solid velocity approximation $(\vect{v}_{s,h})_2$.}}
    \label{fig:ex2anisotropic}
\end{figure}   
%


{
Figure \ref{fig:ex2isotropic} illustrates the numerical solution for the isotropic sandstone. Our HDG method accurately resolves the complex wavefield, specifically identifying the fast P-wave and the dissipative slow P-wave characteristic of Biot's theory. The wave patterns and arrival times are in excellent agreement with the physical properties of the media, as documented in \cite{Alkhimenkov:2021}.}

{In the anisotropic case (Figure \ref{fig:ex2anisotropic}), the simulation captures the distinct directional dependence of the solid velocity components. The HDG formulation, by employing a symmetric stress approximation, ensures that the anisotropic wave surfaces are represented without non-physical spurious oscillation, a common challenge in poroelastic simulations.}

\subsection*{Example 3: Wave propagation on heterogeneous poroelastic media }
{In our final numerical experiment, we extend the application of our method to a heterogeneous poroelastic medium. Simulating wave propagation across material interfaces---specifically a boundary separating isotropic shale and sandstone---presents significant numerical challenges due to the complex reflection and transmission of primary and secondary waves. This example illustrates the robustness of our HDG formulation in handling abrupt changes in material properties within complex geophysical domains.} The computational domain is the rectangle $\Omega= [0,1500]\times [0,1400]$ and the boundary between the two materials is located at the horizontal line $y=700$. Below this line, the material is (isotropic) shale, and above (isotropic) sandstone. We consider the wave propagation of an initial condition set as a Gaussian pulse centered at the point $(x,y) = (750,900)$ in the sandstone material. Zero Dirichlet boundary condition is considered. 


{
The simulation in Figure \ref{fig:ex3heterogeneous} highlights the robustness of the HDG scheme in the presence of material discontinuities. The wavefield transitions smoothly across the interface at $y=700$, with clearly defined reflected and transmitted phases. Notably, the numerical solution remains stable and free of spurious oscillations at the boundary between shale and sandstone, demonstrating that the hybridization variables effectively communicate the interface conditions.}
\begin{figure}[ht]
    \centering
    \includegraphics[width=0.49\linewidth]{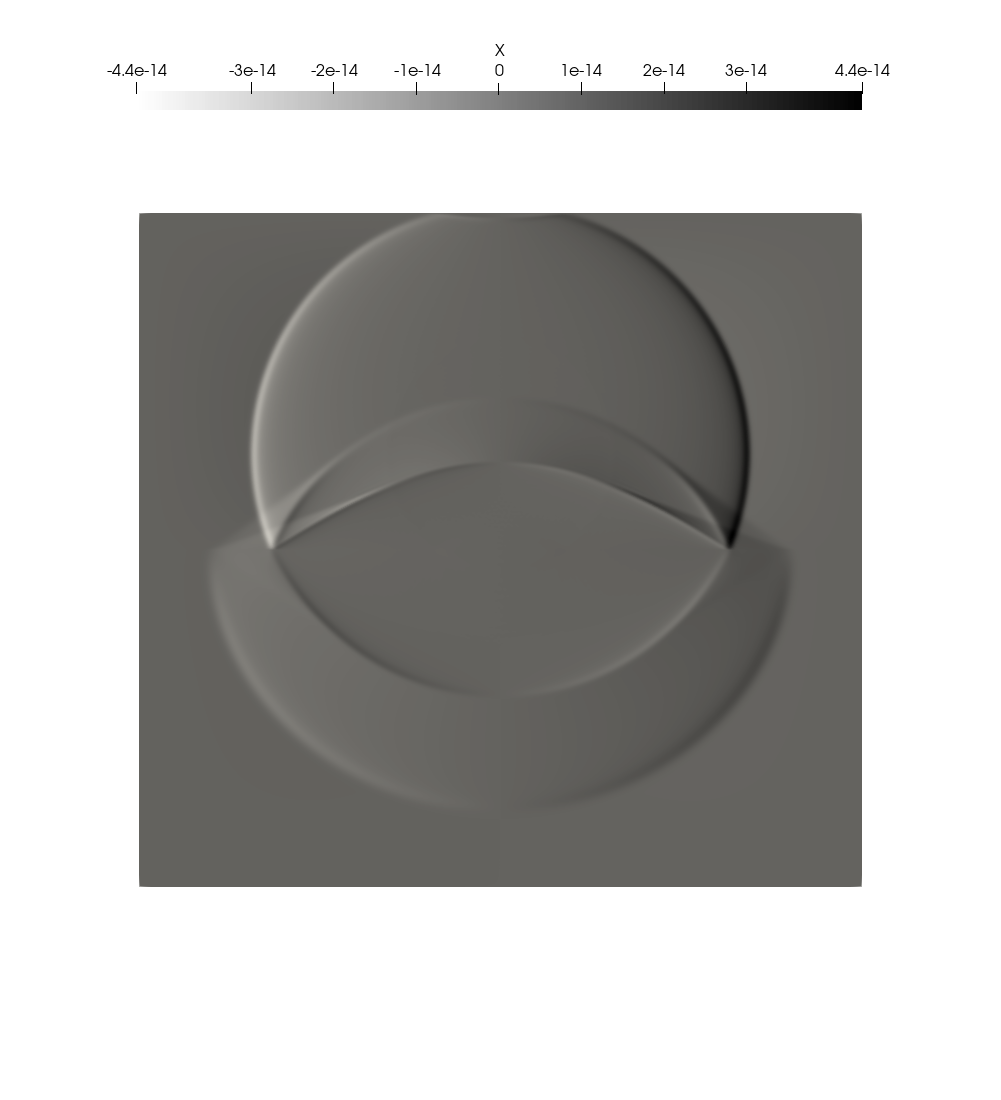}
    \includegraphics[width=0.49\linewidth]{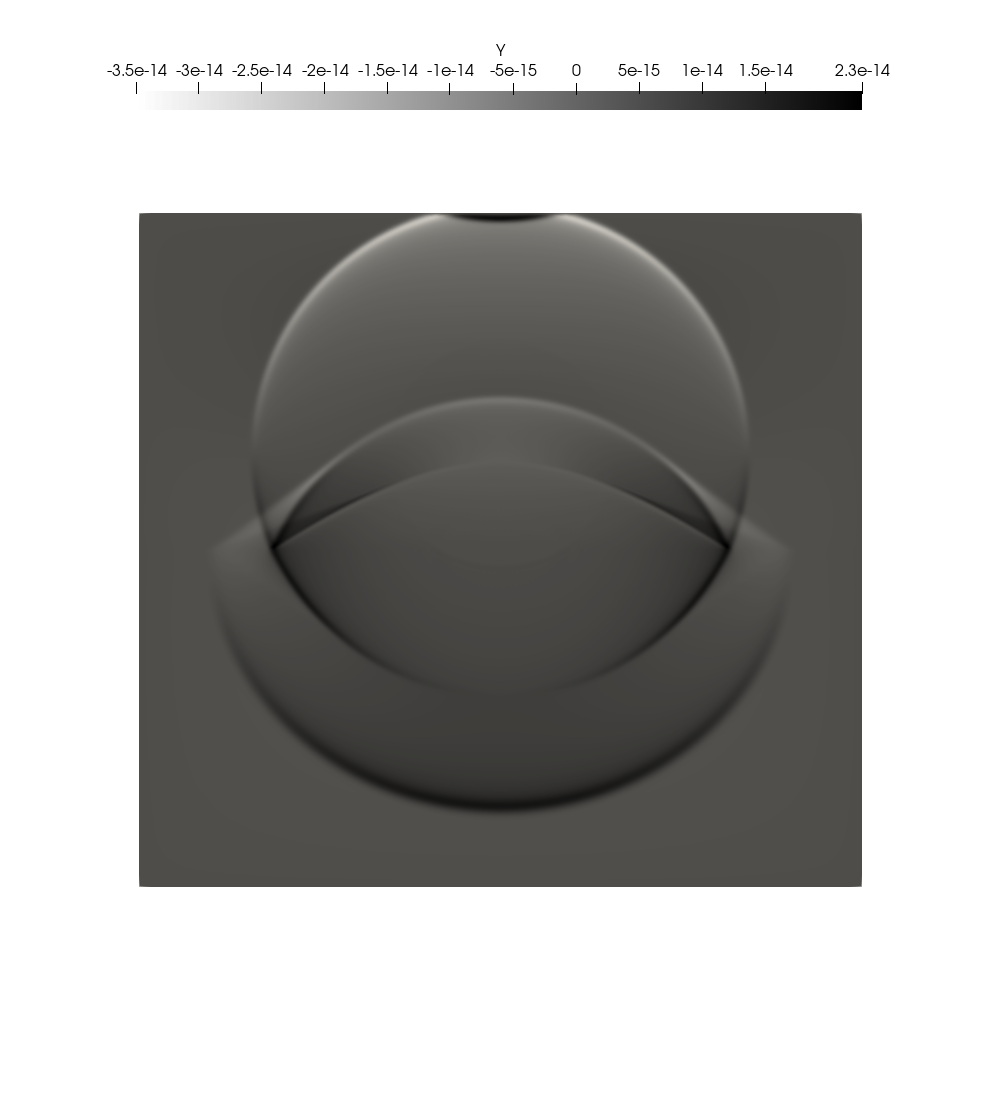}
    \caption{Example 3. Simulation of a Gaussian pulse on heterogeneous isotropic media. Left: First component of the solid velocity approximation $(\vect{v}_{s,h})_1$. Right: Second component of the solid velocity approximation $(\vect{v}_{s,h})_2$. }
    \label{fig:ex3heterogeneous}
\end{figure}

\section{Conclusions}
\label{sec:conclusion}

In this paper, we developed new HDG methods for poroelastic wave equations with error analyses of semidiscrete and fully discrete schemes. In our error analyses, convergence rates of the errors are optimal, except for the solid displacement variable; however, we also observe optimal convergence rates for the solid displacement in numerical experiments. Numerical experiments with more physically realistic settings are consistent with the numerical results of some previous studies of poroelastic wave equations. Future research directions can be applications to more physically relevant settings and the development of efficient solver algorithms.

\providecommand{\bysame}{\leavevmode\hbox to3em{\hrulefill}\thinspace}
\providecommand{\MR}{\relax\ifhmode\unskip\space\fi MR }
\providecommand{\MRhref}[2]{%
  \href{http://www.ams.org/mathscinet-getitem?mr=#1}{#2}
}
\providecommand{\href}[2]{#2}

\end{document}